\documentclass[11pt,reqno,a4paper]{amsart}
\usepackage[T1]{fontenc}    
\usepackage[utf8]{inputenc} 
\usepackage{lmodern}
\usepackage{amsthm,amsmath,amsfonts,amssymb}
\usepackage{bbm}
\usepackage[foot]{amsaddr}
\usepackage[hidelinks]{hyperref}
\numberwithin{equation}{section}

\theoremstyle{plain}
\newtheorem{thm}{Theorem}[section]
\newtheorem{lem}[thm]{Lemma}

\theoremstyle{definition}
\newtheorem{defn}[thm]{Definition}
\newtheorem{rem}[thm]{Remark}
\newtheorem{rems}[thm]{Remarks}

\newcommand{\E}{{\mathbb{E}}}
\newcommand{\N}{{\mathbb{N}}}
\newcommand{\R}{{\mathbb{R}}}

\renewcommand{\P}{{\mathbb{P}}}

\newcommand{\1}{\mathbbm{1}}  
\newcommand{\T}{{\mathbf T}}
\newcommand{\G}{{\mathbf G}}

\renewcommand{\complement}{{\mkern1mu\mathsf{c}}}

\DeclareMathOperator{\prog}{prog}

\DeclareMathOperator{\AEevent}{{\text{\upshape\AE}}}
\newcommand{\cAE}{c_{\!\AEevent\,}}

\renewcommand{\le}{\leqslant}
\renewcommand{\ge}{\geqslant}
\renewcommand{\leq}{\leqslant}
\renewcommand{\geq}{\geqslant}

\let\emptyset\myemptyset

\DeclareMathOperator{\e}{\mathrm{e}}
\DeclareMathOperator{\tr}{tr}
\DeclareMathOperator{\essinf}{ess\,inf}
\DeclareMathOperator{\id}{1}

\providecommand{\wtilde}{\widetilde}
\providecommand{\what}{\widehat}

\newcommand{\ER}{Erd\H{o}s--R\'enyi}

\usepackage{mathtools} \mathtoolsset{centercolon}   
\usepackage{multirow}

\usepackage{enumitem}

\newlist{longlist}{enumerate}{1}
\setlist[longlist]{label=(\roman*), leftmargin=*, nosep, topsep=0pt}

\begin{document}
	
		\title[Return probability on Bienaym\'e--Galton--Watson trees]{Return probability on Bienaym\'e--Galton--Watson trees and spectral asymptotics of sparse \\ Erd\H{o}s--R\'enyi random graphs}

		\author[Markus Heydenreich]{Markus Heydenreich$^1$}
		\email[Markus Heydenreich]{$^1$markus.heydenreich@uni-a.de}
		\address{\normalfont $^1$Institut für Mathematik, \textsc{Universität Augsburg}, Universitätsstra\ss{e} 14, 86159 Augsburg, Germany}
		\author[Peter M\"uller]{Peter M\"uller$^2$}
		\email[Peter M\"uller]{$^2$mueller@lmu.de}
		\author[Sara Terveer]{Sara Terveer$^3$}
		\email[Sara Terveer]{$^3$terveer@math.lmu.de}
		\address{\normalfont $^{2,3}$Mathematisches Institut, \textsc{Ludwig-Maximilians-Universit\"at M\"unchen}, Theresienstra\ss{e} 39, 80333 M\"unchen, Germany}
		\thanks{This work was partially supported by the Deutsche Forschungsgemeinschaft 
			(DFG, German Research Foundation) -- TRR 352 ``Mathematics of Many-Body Quantum Systems and Their Collective Phenomena" -- Project-ID 470903074.}

		\begin{abstract}
	We derive an upper bound for the annealed return probability of the simple random walk on supercritical Bienaym\'e--Galton--Watson trees. The bound decays subexponentially in time $t$ with $t^{1/3}$ in the exponent. It is valid for all offspring distributions with a finite first moment and is optimal whenever the offspring distribution does not exclude leaves or linear pieces in the tree. This solves completely the cases left open by Piau [Ann.\ Probab.\ \textbf{26}, 1016--1040 (1998)]. 
    A new feature of our proof is a far-reaching flexibility in the location of regions with bad isoperimetric properties in the tree. It allows to efficiently treat general offspring distributions and is gained from the joint consideration of the random tree and the random walk as it is inherent under the annealed measure.
    
    In the special case of a Poissonian offspring distribution we apply the upper bound for the annealed return probability to deduce a Lifshits tail for the empirical eigenvalue distribution of the graph Laplacian on supercritical Erd\H{o}s--R\'enyi random graphs with finite mean degree.
\end{abstract}

	\subjclass[2020]{05C81;60K37;60J80}
	\keywords{Bienaym\'e--Galton--Watson tree, simple random walk, return probability, heat kernel, Erd\H{o}s--R\'enyi random graph, empirical eigenvalue distribution, Lifshits tails}

\maketitle


\section{Introduction}
	
In this section, we introduce our basic notation, present the two main results and discuss their 
relation to results in the literature.

\subsection{Return probability on Bienaym\'e--Galton--Watson trees} 
We consider a Bienaym\'e--Galton--Watson branching process 
\cite{Harris63, AthreyaNey04, LyonsPeres16, Hofstad17} whose offspring distribution is given by 
the $\mathbb N_0$-valued probability measure $\mu$. 
The associated Bienaym\'e--Galton--Watson family tree (BGWT) can be viewed as a probability space
$(\mathbb T,\mathcal F,\mathbb P_\mu^*)$ of rooted labelled tree graphs $\mathbf T\in\mathbb T$ 
with the probability measure $\mathbb P_\mu^*$ arising canonically from the offspring distribution 
$\mu$. In particular, $\mu(k)$ is the probability that a given vertex in $\mathbf T$ has 
$k \in\mathbb{N}_0$ children.

The labelling of the tree graphs is conducted using \emph{Ulam--Harris vertex labelling}, 
see e.g.\ \cite[Sect.\ VI.2.1]{Harris63} or \cite[Sect\ 1.1]{legall05}): 
We consider the countable set $\mathcal X := \bigcup_{k=0}^\infty \N^k$ of all possible labels, where
$\N^0: =\{o\}$ indicates the root and the label $x=(x_1,\dots,x_k)\in \mathcal X$, $k\in\N$, denotes 
child number $x_k$ of child number $x_{k-1}$ of $\ldots$ of child number $x_1$ of the root. 
Each labelled rooted tree $\T\in\mathbb T$ is identified with a subset 
$\mathbf T\subset \mathcal X$, but not each subset of $\mathcal X$ corresponds to a tree. 
The $\sigma$-algebra $\mathcal F$ is generated by the maps 
$\big\{\mathbb T \ni \T \mapsto \1_\T(x)\big\}_{x\in\mathcal X}$, where $\1_M$ denotes the indicator function of a set $M$. 

Throughout, we shall assume that the offspring distribution $\mu$ is non-degenerate and that its first moment exists and 
satisfies 
\begin{equation}
	\lambda\coloneqq\lambda(\mu)\coloneqq\sum_{k\in\mathbb N}k\,\mu(k)>1.
\end{equation}
This implies that the BGWT is supercritical, see e.g.\ \cite[Sect.\ 5.1]{LyonsPeres16}), 
in particular, its extinction probability satisfies
\begin{equation}
    \label{eq:ext-prob-def}
    \Lambda\coloneqq \Lambda(\mu)\coloneqq \mathbb P_\mu^*(|\mathbf T|<\infty)<1.
\end{equation}
Here, and subsequently, we write $|\pmb\cdot|$ to denote the cardinality of a set, and, in slight abuse 
of notation, $|\mathbf T|$ stands for the cardinality of the vertex set of 
$\mathbf T$, that is, the total progeny of the tree. Unless $\mu(0)=0$, which is the case for a BGWT without 
leaves, there is always a positive probability that the BGWT ends up being finite. In this case, we condition 
on the event that the BGWT is infinitely large and denote the resulting conditional measure by 
\begin{equation}
    \label{eq:cond-meas-def}
	\mathbb P_\mu(\,\pmb\cdot\,) \coloneqq \mathbb P_\mu^*\big(\,\pmb\cdot\,\mid |\mathbf T|=\infty\big).
\end{equation}
The expectation with respect to $\mathbb P_\mu$ is denoted by $\E_\mu$, the one with respect to 
$\mathbb P_\mu^*$ by $\E_\mu^*$.
		
We now consider the simple random walk $(X_t)_{t\in\mathbb{N}_0}$ in discrete time $t$ 
on the rooted tree $\mathbf T$ starting at the root $o$. In each step, the random walk moves to a 
uniformly chosen neighbouring vertex in $\mathbf T$ of its current vertex position. The law of 
$(X_t)_{t\in\mathbb{N}_0}$ on the space of trajectories $\T^{\N_{0}}$ is denoted by $P_o^\mathbf T$. 
The Kesten-Stigum theorem \cite{KesteStigu66} shows that under $\mathbb P_\mu$ the number of 
children grows a.s.\ exponentially in the number of generations. Hence the effective resistance 
between the root and infinity vanishes, and thus the random walk is transient \cite{LyonsPeres16}. 
        
We are interested in the annealed return probability to the root at time $t$ of the simple 
random walk on supercritical BGWT conditioned on non-extinction. 
We recall that a random walk on a tree can only return to its starting point after an even number of steps.
In the case $\mu(0)=\mu(1)=0$, where every vertex has at least two children, Piau \cite{Piau98} proved the exponential bounds
    \begin{equation}
        \exp(-c't)
        \leq \E_{\mu}\big[ P_o^{\mathbf T}(X_{t}=o)\big]
        \leq\exp(-ct)\qquad\text{
for even $t\in\N_0$.}\label{eq:piau1}
    \end{equation}
However, if $\mu(0)> 0$ or $\mu(1)>0$, then along the lines of Piau \cite{Piau98}, one obtains the subexponential lower bound 
    \begin{equation}\label{eq:LowerBd}
        \exp\big(-c't^{1/3}\big)
        \leq \E_{\mu}\big[ P_o^{\mathbf T}(X_{t}=o)\big]\qquad\text{
for even $t\in\N_0$.}
    \end{equation}

Our primary result is the complementing upper bound, which is optimal in the power $t^{1/3}$ in the exponent.
		
\begin{thm}\label{thm:returnprobability}	
	Assume that the first moment $\lambda$ of the offspring distribution $\mu$ exists 
    and satisfies  $\lambda>1$.	Then, there exists a constant $c>0$ such that 
	\begin{equation}  \label{eq:thm:returnprobability}
		\E_\mu\big[P_o^{\mathbf T}(X_t=x)\big]
        \leq \exp\big(-ct^{1/3}\big)
	\end{equation}
	for every $t\in\mathbb{N}_0$ and every $x\in\mathcal X$.
\end{thm}

\begin{rems}
    \begin{longlist}
    \item
        The statement of Theorem~\ref{thm:returnprobability} improves previously
        established upper bounds: 
        Piau \cite[Thm.~2]{Piau98} employed a regeneration-time analysis to prove the 
        optimal power in the exponent for leaf-free trees, as well as weaker upper 
        bounds for general offspring distributions:
        \begin{align}
            \begin{array}{rll}
        \text{(a)}&\text{ if }\mu(0)=0:&\E_\mu\big[P_o^{\mathbf T}(X_t=o)\big]\leq\exp\big(-ct^{1/3}\big).\\[.5ex]
		\text{(b)}&\text{ if $\mu$ is finitely supported: }\,&\E_\mu\big[P_o^{\mathbf T}(X_t=o)\big]\leq\exp\big(-ct^{1/5}\big).\\[.5ex]
		\text{(c)}&\text{ for general }\mu:&\E_\mu\big[P_o^{\mathbf T}(X_t=o)\big]\leq\exp\big(-ct^{1/6}\big).
            \end{array}
            \label{eq:piau2}
        \end{align}
        The upper bounds in (b) and (c) rely on detailed estimates how (large) finite side trees cause a slowing down of the random walk as compared to its motion on the backbone of the tree. 
    \item
        Vir\'ag \cite{Virag00} pioneered an alternative approach to return and transition probabilities 
        via anchored isoperimetric properties of the underlying graph. He proves an upper bound with an exponential decay in $t^{1/3}$ for the transition probability of the simple random walk on a general deterministic graph of bounded degree and positive anchored expansion constant, see \eqref{T-exp-const} for the definition. 
        Together with the almost-sure positivity of the anchored expansion constant for supercritical BGWTs \cite[Cor.\ 1.3]{ChenPeres}, see also \cite[Thm.\ 6.52]{LyonsPeres16}, this yields exponential decay in $t^{1/3}$ of the \emph{quenched} transition probability for the simple random walk on supercritical BGWTs with finitely supported $\mu$. However, these results do not imply a bound for the \emph{annealed} situation, which we consider here. This is because the initial time from which Vir\'ag's bound applies depends on the underlying graph in an intricate way.  
     
        To overcome this difficulty, the recent paper \cite{MullerStern24} develops a probabilistic version of Vir\'ag's approach. This yields the optimal power for BGWTs with finitely supported offspring distributions and an improvement of Piau's upper bound for very fast-decaying $\mu$:
		\begin{align}
            \label{eq:mullerstern}        
		      \text{\begin{tabular}{rll}
		          (b')&if $\mu$ is finitely supported: 
                    &$\E_\mu\big[P_o^{\mathbf T}(X_t=o)\big]\leq\exp\big(-ct^{1/3}\big)$.\\
		          (c')&if $\mu$ has very fast decay, i.e.&\multirow{2}{*}{$
                        \E_\mu\big[P_o^{\mathbf T}(X_t=o)\big]
                        \leq \exp\big(-ct^{(1 -2/k) /3}\big)$.}\\
                    & $\mu(j)\leq c_1\cdot \exp(-c_2 j^k)$, $k>2$:
	        \end{tabular}}
	    \end{align}
    \item 
        Theorem~\ref{thm:returnprobability} establishes the optimal exponent $t^{1/3}$ in 
        full generality and closes the gap left out by previous results. 
        The proof, which turns out to be surprisingly short and simple, is isoperimetric 
        in spirit and uses ideas from \cite{Virag00} and \cite{MullerStern24}.
        A central new feature in the proof of Theorem~\ref{thm:returnprobability} in 
        comparison to earlier papers is a far-reaching flexibility in the 
        location of regions in the tree with bad isoperimetric properties. 
        It is gained from the joint consideration of the random 
        tree and the random walk as it is inherent under the annealed measure. It is this 
        flexibility which allows to handle general offspring distributions.
    \item
        The Poisson case $\mu=\mathrm{Poi}(\lambda)$ for the offspring distribution
        is of particular importance, because such BGWTs arise as the local weak limit of 
        Erd\H{o}s--R\'enyi random graphs, see e.g.\ \cite[Theorem 2.18]{Hofstad24}.          
        Theorem~\ref{thm:returnprobability} covers this case for the first time and provides 
        the exact exponent $t^{1/3}$. We present an application of Theorem~\ref{thm:returnprobability} towards 
        Erd\H{o}s--R\'enyi random graphs in Section~\ref{sec:thm-lifshits}.
    \item 
        It is an open problem to get refined estimates which establish
        \eqref{eq:thm:returnprobability} and \eqref{eq:LowerBd} with the same constant
        multiplying $t^{1/3}$ in the exponent asymptotically as $t\to\infty$. 
    \item
        The bound of Theorem~\ref{thm:returnprobability} carries over to the two canonical variants of 
        continuous-time random walks on supercritical BGWTs: 
        the one with i.i.d.\ waiting-time clocks attached to the vertices, which is generated by the normalised Laplacian,  
        and the other one with i.i.d.\ waiting-time clocks attached to the edges, which is generated by the Laplacian. 
        We explore upper and lower bounds for the annealed return probabilities of these continuous-time random walks in a forthcoming publication \cite{HeMuTe-inprep}.
    \item 
        Interestingly, the case of \emph{biased} random walks (i.e., random walks with 
        drift) on BGWTs has attracted significant attention in the literature, e.g.\ 
        \cite{Aidekon14,BenArousFribergh16, DemboEtAl02,LyonsPemantlePeres96}. Their 
        results are mostly concerned with the effective speed of the walk.
    \item 
        Almost-sure positivity of the speed of the random walk on an exponentially 
        growing stationary random graph was studied in \cite{BePaPf18}. This paper also 
        uses the anchored expansion approach, and the authors were able to effectively trade the bounded-degree assumption of \cite{Virag00} for the assumptions of stationarity and exponential growth.
    \end{longlist}
\end{rems}

Apart from answering a natural probabilistic question, Theorem~\ref{thm:returnprobability}
has also relevance in the spectral theory of random graphs, as we explain in the next 
section.

\subsection{Empirical eigenvalue distribution of the Laplacian on finite-mean-degree Erd\H{o}s--R\'enyi random graphs} 
\label{sec:thm-lifshits}

The ensemble $\mathcal{G}(N, p(N))$ of \ER{} random graphs, as introduced by \cite{Gil59}, 
consists of all simple, undirected graphs $\mathbf G^{(N)}$ with 
$N\in\mathbb{N}\setminus\{1\}$ vertices enumerated by $1$ to $N$
and randomly chosen edges between all $\binom{N}{2}$ pairs of vertices  \cite{Bol01, ErRe60}. 
These edges are drawn independently with 
probability $p\coloneqq p(N) \in {}]0,1[$. 
We refer to the choice $p(N ) \coloneqq \lambda/N$ with
$\lambda \in \,]0, \infty[\,$ as \emph{finite-mean-degree \ER{} random graphs} because, asymptotically for large $N$, the average total number of edges grows linearly with the number of vertices $N$, and the average degree of any given vertex equals $\lambda$. 
Finite-mean-degree \ER{} random graphs $\mathbf G^{(N)}$ exhibit a structural, percolation-like phase transition at $\lambda=1$.
In contrast to the subcritical regime $\lambda \in {}]0,1[\,$, the supercritical regime $\lambda \in {}]1,\infty[\,$ features an emerging
unique \emph{giant cluster} $\mathbf C_{\max}^{(N)}$ as $N\to\infty$ \cite{Bol01,ErRe60} which absorbs a positive fraction of all 
vertices. Here, we say that a subgraph of $\G^{(N)}$ is a cluster, if it is a connected component of $\G^{(N)}$. By
convention, we want to include isolated vertices as one-vertex clusters in this notion, too. 
The finite-mean-degree case is of great interest in various applications of \ER{} random graphs in physics and other disciplines. However, from a spectral point of view, it is mathematically considerably more challenging and therefore much less understood than cases in which $Np(N) \rightarrow \infty$.  
By way of example, we mention the papers \cite{MR4255147, MR4328063, MR4375917, MR4515695, MR4691859, MR3945756, MR2964770, MR3098073, MR4288336, MR4021251} to illustrate the tremendous progress achieved over the last decade for the spectral properties of the $N\times N$-self-adjoint adjacency matrix $A^{(N)} \coloneqq (A^{(N)}_{jk})_{1 \le j,k \le N}$,
\begin{equation}
    A^{(N)}_{jk} \coloneqq \left\{\begin{array}{@{}l@{\;\;}l}1, & \text{if vertices $j$ and $k$ are 
        connected by an edge in $\G^{(N)}$,} \\	0, & \text{otherwise,} \end{array} \right.
\end{equation}
of \ER{} random graphs with $Np(N) \rightarrow \infty$ and how the frontier was pushed towards increasing sparsity. In comparison, the rigid constraints of vanishing row and column sums hamper studies \cite{MR2206341, MR4058984, MR2967963} of the spectral properties of the self-adjoint graph Laplacian 
\begin{equation}
    \label{eq:Laplace-def}
    \Delta^{(N)} \coloneqq D^{(N)} - A^{(N)}
\end{equation} 
or related operators even in the simpler cases when $Np(N) \rightarrow \infty$. Here,
\begin{equation}
    D^{(N)} \coloneqq \mathrm{diag}\big(\deg_{\G^{(N)}}(1), \ldots, \deg_{\G^{(N)}}(N)\big)
\end{equation}
denotes the diagonal $N\times N$-matrix of the vertex degrees in $\G^{(N)}$.
We emphasize that the powerful methods and techniques which prove useful when $Np(N)\to \infty$ mostly do not apply in the finite-mean-degree case of \ER{} random graphs. The few results which were previously known in the finite-degree case will be discussed after Theorem~\ref{thm:lifshits}. 
The quantity of our interest here is defined in 
\begin{lem}
    \label{lem:DOS-exist}
    Let $\lambda>0$. Then there exists a Borel probability measure $\sigma_\lambda$ on $\R$, 
    the \emph{density-of-states measure} (or averaged empirical eigenvalue distribution) of the 
    graph Laplacian on finite-mean-degree \ER{} random graphs, such that the sequence 
    $\big(\sigma_\lambda^{(N)}\big)_{N}$ of eigenvalue counting measures, defined by  
 	\begin{equation}
		\label{eq:DOS-def}
        \sigma_\lambda^{(N)}(B) \coloneqq \frac{1}{N} \; \E^{(N)}_{\lambda} \Big[ \mathrm{tr}_{\mathbb C^{N}}   \big\{ \1_{B}(\Delta^{(N)}) \big\} \Big] 
	\end{equation}   
    for $N\in\N\setminus\{1\}$ and any Borel set $B \subseteq \R$, 
    converges weakly to $\sigma_\lambda$ as $N\to\infty$.
    In \eqref{eq:DOS-def}, $\E^{(N)}_{\lambda}$ denotes the average over the \ER{} ensemble 
    $\mathcal G(N,\lambda/N)$, and the trace $\mathrm{tr}_{\mathbb C^{N}}$ of the spectral projection 
    $\1_{B}(\Delta^{(N)})$ counts the number of eigenvalues of $\Delta^{(N)}$, including their multiplicities, in the set $B$.
\end{lem}

The traditional proof of Lemma~\ref{lem:DOS-exist} argues via moments, see e.g.\ Lemma~2.3 in \cite{KiKhMu06}.
A more recent (and fruitful) approach is based on the fact that the \emph{local weak limit} 
(or Benjamini--Schramm limit) of  the ensemble of \ER{} random graphs $\mathcal G(N,\lambda/N)$ is given by the ensemble of Bienaym\'e--Galton--Watson trees with Poissonian offspring distribution $\mu=\mathrm{Poi}(\lambda)$ with mean $\lambda$, see e.g.\ \cite[Thm.\ 2.18]{Hofstad24} or \cite[Sect.\ 2.2]{Bor17}. 
In order to explore this direction, we recall  \cite{Jor08,KeLe10} that for $\P^*_{\mathrm{Poi}(\lambda)}$-almost every tree 
$\T$, the graph Laplacian $\Delta_\T$ is an essentially self-adjoint operator when defined on the 
subspace of finitely supported vectors $ f\coloneqq \big(f(x)\big)_{x\in\T}$ in the Hilbert space $\ell^2(\T)$
of complex-valued, square-summable sequences over the vertex set of $\T$. Given such $f$, it acts as 
\begin{equation}
    \label{eq:Delta-unbounded}
    (\Delta_\T f)(x) \coloneqq \deg_\T (x) f(x) - \sum_{\substack{y\in \T \,:\;  \{x,y\} \\ \text{is an edge in } \T}} f(y) \qquad \text{for every $x\in\T$}.
\end{equation}
(This is only relevant for $\lambda>1$ because if $\T$ is finite, then the definition of the graph Laplacian by \eqref{eq:Delta-unbounded} is equivalent to the one in \eqref{eq:Laplace-def} with 
$\G^{(N)}=\T$.) Therefore, the expectation in \eqref{BJ-limit} of the scalar product $\langle\pmb\cdot,\pmb\cdot\rangle$ on $\ell^2(\T)$ with the canonical basis vector $\delta_o \in \ell^2(\T)$, where, more generally $\delta_x(y) \coloneqq \1_{\{x\}}(y)$ for every $x,y\in\T$, is well defined, and the equality 
\begin{equation}
    \label{BJ-limit}
	\sigma_\lambda = \E^*_{\mathrm{Poi}(\lambda)} \big[\langle \delta_{o}, \1_{\bullet}(\Delta_{\T}) 
    \delta_{o}\rangle \big]
\end{equation}
follows from the uniqueness of weak limits along the lines in the proof of \cite[Prop.\ 2.2 (ii)]{Bor17}, which is carried out for the adjacency operator there. It works equally well for the Laplacian because all that matters is essential self-adjointness on the subspace of finitely supported vectors and locality of the operator. 

According to \cite{Salez20}, an explicit description of $\sigma_\lambda$ ``seems out of reach at present, and our understanding remains extremely limited.'' 
The second main result of this paper establishes a Lifshits-tail behaviour of 
$\sigma_\lambda$ at the lower spectral edge $E=0$ of the Laplacian's spectrum. It answers an open question posed 20 years ago in \cite{KiKhMu06}.

\begin{thm}
	\label{thm:lifshits}
	Let $\lambda>1$. Then there exist constants $E_{0}>0$ and $c' \ge c >0$ such that 
	\begin{equation}
        \label{eq:thm:lifshits}
		  \exp\big(-c'E^{-1/2}\big) \le \sigma_{\lambda}(\,]0,E]) \le \exp\big(-cE^{-1/2}\big) 
	\end{equation}
	for every $E \in \,]0, E_{0}]$. 
\end{thm} 

\begin{rems}
    \begin{longlist}
    \item
        The statement of Theorem \ref{thm:lifshits} also holds in the subcritical phase 
        $\lambda \in \,]0,1[\,$, as was proved in \cite{KiKhMu06}. Its validity in the supercritical 
        phase requires additional control of the contribution of the giant cluster, which we obtain 
        from Theorem~\ref{thm:returnprobability}. The very fast decay with $E^{-1/2}$ in the exponential 
        as $E\downarrow 0$ reflects that the dominant contribution of small Laplacian eigenvalues is due 
        to long linear pieces in the graphs. The critical case $\lambda=1$ has not been studied so far. 
    \item
	    The high-energy asymptotics 
        \begin{equation}
            \lim_{E\to\infty} \frac{- \ln\big(1- \sigma_\lambda(\,]-\infty,E])\big)}{E \,\ln E} = 1
        \end{equation}
        is known for every $\lambda>0$ \cite{Sattler10}. Morally, very large eigenvalues $E$ of the 
        Laplacian are probabilistically suppressed as $1/E!$ (understood as Euler's Gamma function) because 
        they require large vertex degrees of the order of $E$, and the degree distribution is Poissonian.
    \item 
        The density-of-states measure $\sigma_\lambda$ possesses a dense countable set of atoms at algebraic 
        integers, as follows from the early physics paper \cite{ChChFr86}. The atom at zero
        \begin{equation}
            \label{eq:conn-comp}
            \sigma_\lambda(\{0\}) = \E^*_{\mathrm{Poi}(\lambda)} \big[\1_{\{|\T| < \infty\}} \, 1/|\T| \big]
            = \lim_{N\to\infty} \sigma_\lambda^{(N)}(\{0\}) 
        \end{equation}
        describes the limiting mean number density of connected components in finite-mean-degree \ER{} graphs because each connected 
        component of $\G^{(N)}$ contributes exactly one eigenvalue $0$ to the spectrum of the Laplacian. 
        The second equality in \eqref{eq:conn-comp}, which is not covered by the definition weak 
        convergence of measures, follows from \cite[Cor.\ 2.21]{Hofstad24}. 
        A detailed mathematical analysis of the set of atoms of $\sigma_{A,\lambda}$---the corresponding measure for the adjacency operator---is carried 
        out in \cite{Salez20}, see Cor.~3 there and subsequent comments. 
        The strength of the atom $\sigma_{A,\lambda}(\{0\})$ was already studied in \cite{BoLeSa11}. 
    \item 
        The measure $\sigma_{A,\lambda}$ has a non-trivial continuous part if and only if $\lambda >1$ 
        \cite[Cor.\ 1.7]{BoSeVi17}. Its mass around $0$ as determined by 
        $g_\lambda(E)\coloneqq\sigma_{A,\lambda}([-E,E] \setminus \{0\})$ is of order $o(E^2)$ as 
        $E\downarrow 0$ for $\lambda <\e$ and obeys $\limsup_{E\downarrow 0} g_\lambda(E) /E >0$
        for $\lambda >\e$ \cite{CosteSalez21}. This may be taken as an indicator for an absolutely 
        continuous component of the density-of-states measure (and even delocalisation) above the critical 
        mean degree $\lambda=\e$. Spectral delocalisation in its proper sense, that is, the 
        $\P_{\mathrm{Poi}(\lambda)}$-a.s.\ existence of an absolutely continuous component of 
        $\langle\delta_o, \1_\bullet(A_\T)\delta_o\rangle$ for sufficiently large $\lambda$,
        is shown in \cite{ArBo24}.  
    \item        
        Finally, we comment on stronger notions of convergence for empirical eigenvalue distributions in the local weak limit. In \cite{BoLeSa11} the (non-averaged) empirical eigenvalue distributions
        $N^{-1}\mathrm{tr}_{\mathbb C^{N}} \big\{ \1_{\bullet}(A^{(N)}) \big\}$ are shown to converge weakly to $\sigma_{A,\lambda}$ \emph{in probability}. Inspecting the proofs of Cor.\ 12 and Prop.\ 11 there, 
        this result is seen to hold for the Laplacian as well. It is not obvious to us, however, whether the following two results transfer from the adjacency operator to the Laplacian: Weak convergence of 
        $N^{-1}\mathrm{tr}_{\mathbb C^{N}} \big\{ \1_{\bullet}(A^{(N)}) \big\}$ to 
        $\langle\delta_o, \1_\bullet(A_\T)\delta_o\rangle$ almost surely \cite{BoLeSa11, CampbellORourke24} and 
        uniform convergence of the distribution function of $\sigma_{A,\lambda}^{(N)}$ to 
        that of $\sigma_{A,\lambda}$ as $N\to\infty$ \cite[Thm.\ 2.5]{Bor17}.
    \end{longlist}
    \label{rem:lif-thm}
\end{rems}

\subsection{Organisation of the paper}
We establish our notation and some useful properties of 
BGWTs conditioned on survival in Section~\ref{sec:bgwtsurvival}. In Section~\ref{sec:isolatedcores} we 
identify ``bad'' regions in trees that hamper the analysis of the return probability and collect properties and tools to handle them.
The proof of Theorem~\ref{thm:returnprobability} will be given in Section~\ref{sec:proof}, where we obtain a 
suitable bound for the annealed probability that such a bad region is hit by the random walk. When 
combined with isoperimetric results for walks that do not hit any of these bad regions, the theorem 
follows. The proof of Theorem~\ref{thm:lifshits} is given in Section~\ref{sec:lifshits}. Here, the 
main point is to transfer the result of Theorem~\ref{thm:returnprobability}, which relates to the 
normalised graph Laplacian, to an upper bound for the (standard) graph Laplacian. 

\section{Properties of BGWTs conditioned on survival}\label{sec:bgwtsurvival}

In this section we are going to collect a few technical properties of BGWTs conditioned on survival and introduce necessary technical notation needed in the proof of Theorem \ref{thm:returnprobability}.
We begin by stating that the event for observing a finite but large tree in the supercritical regime of the BGWT is exponentially rare, 
that is, there exists a constant $c_{\mathrm{fin}} >0$ such that 
\begin{equation}
	\mathbb P_\mu^*\big(|\mathbf T|\geq M \mid |\mathbf T|<\infty\big)\leq \exp (-c_{\mathrm{fin}} M),
	\label{eq:finitetreetails}
\end{equation}
for every $M \in\mathbb{N}$, see \cite[Theorem I.13.1]{Harris63} or \cite[Lemma 2]{Piau98}. 

In view of the tree Markov property, it is helpful to utilize a BGWT conditioned on survival as obtained from a branching process with independent offspring generation. However, this requires two different kinds of individuals: 
We assign to each vertex $x\in \T$ a
type $f \in \{S, E\}$ depending on whether $x$ has infinitely many descendants in $\T$ (``\emph{survivor}'', type $S$) or whether it has only finitely many descendants in $\T$ (``\emph{going extinct}'', type $E$), see e.g.\ \cite[Chap.\ I.D.12]{AthreyaNey04} and 
\cite[Sect.\ 5.7]{LyonsPeres16}. In order to generate the BGWT-ensemble under the conditional measure 
$\P_{\mu}$ one starts with the root which is always assigned type $S$. A vertex of type $S$ must have at least one child of type $S$. Its total number $n\in\N$ of enumerated children, among which the ones in the set $\emptyset\neq U\subseteq\{1\dots,n\}$ are of type $S$ and the remaining ones of type $E$, follows the distribution 
\begin{equation}
    \label{eq:S-multitype_offspring_distributions}
		\mu_S(n,U)\coloneqq\mu(n)(1-\Lambda)^{|U|-1}\Lambda^{n-|U|}.  
\end{equation}
On the other hand, the probability of a vertex of type $E$ to have $n\in\N_{0}$ children, which then are necessarily all of type $E$ themselves, is given by
\begin{equation}
    \label{eq:E-multitype_offspring_distributions}
        \mu_E(n)\coloneqq\mu(n)\Lambda^{n-1}.
\end{equation}
We recall the relation $\sum_{n\in\mathbb{N}_0} \mu(n) \Lambda^n = \Lambda$ 
for the extinction probability $\Lambda$ \cite[Thm.\ 3.1]{Hofstad17},
which guarantees the normalisation of the distributions $\mu_{S}$ and $\mu_{E}$. We also recall that the case 
$\Lambda=0$ is not relevant in the above considerations because then conditioning is not needed.

Firstly, we turn to the tree Markov property for \emph{unconditional} BGWT and fix an Ulam--Harris label  $x\in\mathcal X$. Let $F_{x}^{+}$, resp.\ $F_{x}^{-}$, be an event that $x\in\T$ plus further requirements concerning only, resp.\ not, the descendants of $x$ in $\T$. Then the construction of the (single-type) BGWT-ensemble under $\P_{\mu}^{*}$ as a branching process implies the Markov property
\begin{equation}
\label{eq:Markov-star}
\begin{split}
	\P^{*}_{\mu}(F_{x}^{-} \cap F_{x}^{+} \mid x\in\T) 
	&= \P^{*}_{\mu}(F_{x}^{-} \mid x\in\T) \;\P^{*}_{\mu}(F_{x}^{+}  \mid x\in\T)  \\
	&= \P^{*}_{\mu}(F_{x}^{-} \mid x\in\T) \; \P^{*}_{\mu}(F_{o}^{+}),
\end{split}
\end{equation}
where $F_{o}^{+}$ is the event obtained from $F_{x}^{+}$ by shifting it back to the root.
Essentially, conditionally on the existence of a given vertex in the tree, its progeny and ancestry are independent, and its progeny is distributed just as the tree at the root.

An analogous tree Markov property holds for BGWT conditioned on survival: Additionally to the vertex $x\in\mathcal X$, we also need to 
 fix its type $f\in \{S,E\}$. We introduce the notation  $x \in\T_{f}$ as a shorthand for $x\in \T$ and that the type of $x$ in $\T$ is $f$. In analogy to the tree Markov property \eqref{eq:Markov-star} with respect to $\P_\mu^*$, the above construction of the BGWT-ensemble under $\P_{\mu}$ as a two-type-branching-process implies the tree Markov property 
\begin{equation}
	\label{eq:Markov}
	\P_{\mu}(F_{x}^{-} \cap F_{x}^{+} \mid x\in\T_{f}) 
	= \P_{\mu}(F_{x}^{-} \mid x\in\T_{f}) \;\P_{\mu}(F_{x}^{+}  \mid x\in\T_{f})  
\end{equation}
for events $F_{x}^{\pm}$ as in \eqref{eq:Markov-star}, i.e. the progeny and ancestry of a given vertex conditionally on its type are independent. Depending on the type, the rightmost conditional probability in \eqref{eq:Markov} can be reduced as follows
\begin{align}
 \P_{\mu}(F_{x}^{+} \mid x\in\T_{S}) &= \P_{\mu}(F_{o}^{+}), \label{eq:shift-S}\\
 \P_{\mu}(F_{x}^{+} \mid x\in\T_{E}) &= \P_{\mu_{E}}^{*}(F_{o}^{+}) 
 		= \P_{\mu}^{*}(F_{o}^{+} \mid |\T| < \infty). \label{eq:shift-E}
\end{align}

\section{Isoperimetric ratios and consequences}\label{sec:isolatedcores}
        
This is the second section of a preparatory nature to set the stage for the proof of Theorem~\ref{thm:returnprobability} in Section~\ref{sec:proof}. 
        
It is straightforward -- and we will do so in Lemma \ref{lem:complement} below -- to analyse the return probability of a random walk on a graph on which the operator norm of its Markov kernel is bounded away from one. Such a bound can be achieved if the graph has a positive edge-isoperimetric constant, see e.g.\ \cite[Sect.\ 6.1]{LyonsPeres16} or \cite[Sect.~3]{MullerStern24}. 
In turn, this is the case if the graph has no arbitrarily bad
\emph{isolated cores}, that is, subgraphs with arbitrarily small edge boundary as compared to their volume. 

\subsection{Isolated cores}

We follow Vir\'ag \cite[Sect.~3]{Virag00} for the notions introduced in this section but keep the conventions of \cite{MullerStern24} in that we use unweighted volumes $|\pmb\cdot|$. 

\begin{defn}
    Let $q>0$ and a let $\G$ be a locally finite graph with a countable vertex set.
    \begin{longlist}
	\item     
        The $q$-\emph{isolation} 
        of a (possibly empty) finite vertex subset $V \subseteq \mathbf G$ is defined as  
	    \begin{equation}
	       \iota_q(V)\coloneqq q|V|-|\partial V|, 
        \end{equation}
	    where $\partial V\coloneqq \partial_\G V$ denotes the \emph{edge boundary} of $V$ in $\G$. 
	\item 	
        A (possibly empty) finite vertex subset $V$ is called a $q$-\emph{isolated core} if 
	    \begin{equation}
	        \iota_q(W)< \iota_q(V) \qquad
	        \text{for all proper subsets $W\varsubsetneqq V$}.
	    \end{equation}
	\item 	
        The union of all $q$-isolated cores in $\mathbf G$ is denoted by $\mathcal A_q(\mathbf G)$, its connected components are called \emph{$q$-islands}. The complementary vertex set $\mathbf G\setminus \mathcal A_q(\mathbf G)$ is called the \emph{$q$-oceans} of $\mathbf G$.
    \end{longlist}
\end{defn}

\begin{rem}
		\label{delta_q-pos}
		Since $\iota_q(\varnothing)=0$, we have $\iota_q(V)>0$ for every non-empty 
		$q$-isolated core $V$,
        that is, an upper bound for the \emph{isoperimetric ratio}
        \begin{equation}
            \frac{|\partial V|}{|V|} < q
        \end{equation}
        of $V$.
\end{rem}
A central issue in Section~\ref{sec:proof} is to estimate the probability that the random walk visits large $q$-isolated cores. 
		
\subsection{Skipping islands}
Another useful tool in our analysis will be the random walk obtained from the original simple random 
walk on the BGWT by suitably jumping over $q$-islands. Sometimes referred to as an \emph{induced random walk}, its mathematical formulation requires 
weighted graphs $(\mathbf G, w)$, where $w\colon\mathbf G\times\mathbf G\to[0,\infty[\,$ 
is a symmetric edge-weight function.

\begin{defn}
    Let $(\mathbf G,w)$ be a weighted graph with a countable vertex set.
	\begin{longlist}
	\item 
        For $x\in \mathbf G$ we introduce the \emph{vertex weights} 
		\begin{equation}
		    w(x)\coloneqq\sum\limits_{y\in\mathbf G} w(x,y),
		\end{equation}	
        which we assume to be strictly positive and finite. 
	\item 
        We denote by $\langle \cdot,\cdot \rangle_{\mathbf G,w}$ the \emph{$w$-weighted inner product} 
        on the weighted Hilbert space $\ell^2(\mathbf G,w)$ given by 
		\begin{equation}
		    \langle\phi,\psi\rangle_{\mathbf G,w} 
            \coloneqq \sum\limits_{x\in\mathbf G} w(x)\overline{\phi(x)}\psi(x)
		\end{equation}	
        for every $\phi,\psi \in \ell^2(\mathbf G, w)$. 
	\item 
        The \emph{Markov kernel} $\mathbf{P}_{\mathbf G,w}$ acts as a self-adjoint operator on 
        $\ell^2(\mathbf G,w)$ by 
        \begin{equation}
            (\mathbf P_{\mathbf G,w}\phi)(x) 
            \coloneqq\sum\limits_{y\in\mathbf G}\frac{w(x,y)}{w(x)}\phi(y)
        \end{equation}
		for every $\phi\in\ell^2(\mathbf G,w)$ and every $x\in \mathbf G$. 
        It describes a canonical discrete-time random walk $(W_t)_{t\in\mathbb{N}_0}$ on 
        $(\mathbf G, w)$
        whose transition probability from $x$ to $y$ is given by the normalised edge weight 
        $w(x,y)/w(x)$. If $W_0=x$ with probability one, we denote the probability measure of 
        $(W_t)_{t\in\mathbb{N}_0}$ by $P_x^{\mathbf G,w}$.
	\end{longlist}
\end{defn}


In the remainder of this section we will fix $q\in {}]0,1]$ and consider an infinite tree graph 
$\mathbf T$, which is locally finite and has a countable vertex set. 
We will define the random walk induced by the $q$-oceans 
\begin{equation}
    \mathbf T_q \coloneqq \mathbf T\setminus \mathcal A_q(\mathbf T)
\end{equation}
of $\mathbf T$. It jumps over the $q$-islands of $\mathbf T$ in such a way that its probability for reaching $y \in \mathbf T_q$ from  $x \in \mathbf T_q$ without specifying at which time 
is the same as for the simple random walk on $\mathbf T$.

\begin{defn}
    \label{def:Wt}
    \begin{longlist}
    \item  
        We write 
        \begin{equation}
            \label{eq:tau-oc-def}
            \tau_{\mathrm{oc}} \coloneqq \inf\big\{t\in\N:X_t\in \mathbf T_q\big\}\in\N\cup\{\infty\}
        \end{equation}
        for the first hitting time \emph{after zero} of a $q$-ocean by the 
        simple random walk $(X_t)_{t\in\N_0}$ on the tree $\mathbf T$.
    \item 
        The map $w_q \colon\mathbf T_q\times\mathbf T_q \to [0,\infty[\,$, 
        \begin{equation}
            \label{eq:wq-def}
            w_q(x,y)\coloneqq \deg_{\mathbf T}(x)\,P_x^\mathbf T(X_{\tau_{\mathrm{oc}}}=y),
            \qquad x,y\in\mathbf T_q,
        \end{equation}       
        is an edge-weight function, see also Remark~\ref{rem:wq-prop}. 
    \item 
        We write $P_x^{\mathbf T_q} \coloneqq P_x^{\mathbf T_q, w_q}$ for the probability measure of 
        the canonical random walk $(W_t)_{t\in\N_0}$ on the weighted graph $(\mathbf T_q,w_q)$ 
        starting in $x\in\mathbf T_q$. 
        We abbreviate the associated Markov operator on $\ell^2(\mathbf T_q,w_q)$ by 
        $\mathbf P_q\coloneqq \mathbf P_{\mathbf T_q,w_q}$ and note that 
        \begin{equation}
            \label{eq:prob-Markov}
            P_x^{\mathbf T_q}(W_t =y) 
            = \frac{1}{w_q(x)} \; \langle \delta_{x}, \mathbf P_q^{\mkern2mu t} 
            	\delta_{y}\rangle_{\mathbf T_q, w_q} 
        \end{equation}
    \end{longlist}
\end{defn}

\begin{rem}
    \label{rem:wq-prop}
    \begin{longlist}
    \item
    For symmetry and further properties of $w_q$ we refer to \cite[Rem.\ 3.5]{MullerStern24}. 
    In particular, we mention that for every $x,y \in \mathbf T_q$ the inequality  
    \begin{equation}
        \label{eq:wq-wsrw}
        w_q(x,y) \ge \begin{cases} 1, & \text{ if } \{x,y\} \text{ is an edge of }  
                                   \mathbf T \\ 0, & \text{ otherwise} \end{cases}
    \end{equation} 
    holds. It is strict if and only if $x$ and $y$ belong to the outer vertex boundary of 
    one and the same $q$-island of $\mathbf T$. 
    \item 
    The vertex weights of the weighted graph $(\mathbf T_q,w_q)$ satisfy 
    \begin{equation}
        w_q(x) = \deg_{\mathbf T}(x)
    \end{equation}
    for every $x\in \mathbf T_q$ because the probabilities in \eqref{eq:wq-def} add up to one. 
    \end{longlist}
\end{rem}


The proof of the next theorem follows from combining \cite[Lemma 3.6]{MullerStern24} with 
\cite[Thm.\ 6.7]{LyonsPeres16}.
		
\begin{thm}        
    \label{thm:markovkernelbound}
	Let $q \in {}]0,1]$ and let $\mathbf T$ be a locally finite rooted tree with a countably \emph{infinite} vertex set. Then, the Markov kernel $\mathbf P_q$ of the canonical random walk 
    $(W_t)_{t\in\mathbb{N}_0}$ on the weighted graph $(\mathbf T_q, w_q)$ satisfies
	\begin{equation}
	    \|\mathbf P_q\|_{\mathbf T_q,w_q}\leq 1-\frac{q^2}{18},
	\end{equation}	
	where $\| \cdot \|_{\mathbf T_q,w_q}$ denotes the operator norm on the Banach space of 
    bounded linear operators on $\ell^2(\mathbf T_q, w_q)$.
\end{thm}

\subsection{Useful consequences of anchored expansion}

It was proved in \cite[Cor.\ 1.3]{ChenPeres}, see also \cite[Thm.\ 6.52]{LyonsPeres16}, that 
$\mathbb{P}_\mu$-a.e.\ tree $\mathbf T \in\mathbb T$ possesses a positive anchored expansion constant 
\begin{equation}
	\label{T-exp-const}
	\mathbf{i}(\mathbf{T})  \coloneqq  \lim_{n\to\infty}\inf\Bigl\{\frac{|{\partial K}|}{|K|}:
    \;o\in K\subset\mathbf{T}\mbox{ connected}, n\leq |{K}|<\infty\Bigr\} >0.
\end{equation}
This was strenghtened in \cite[Lemma 2.3]{MullerStern24} where it is shown that even
\begin{equation}
    \mathbf{i}_{\mkern3mu\mathbb T} \coloneqq \mathbb P_\mu\text{-}\mkern2mu 
    \underset{\mathbf T \in\mathbb T}{\essinf} \mkern5mu \mathbf{i}(\mathbf T) >0.
\end{equation}
We conclude this preparatory section with two consequences of the anchored-expansion property of supercritical BGWTs.

\begin{lem}
    \label{lem:ms24_cor29}
    Let $q \in {}]0,\mathbf{i}_{\mkern3mu\mathbb T}[\,$. Then, for $\mathbb P_\mu$-almost every 
    $\mathbf T \in\mathbb T$, every $q$-island of $\mathbf T$ has only finitely many vertices and 
    thus is itself a $q$-isolated core of $\mathbf T$.    
\end{lem}

We refer to \cite[Cor.\ 2.9]{MullerStern24} for the proof.

\begin{lem}
    \label{lem:ms24_lem25}
	There exist constants $q \in {}]0,\mathbf{i}_{\mkern3mu\mathbb T}[\,$ and $\cAE>0$ such that for every $s\in \N$ the \emph{anchored-expansion event} 
	\begin{equation}
	    \AEevent_{s} \coloneqq \left\{\mathbf T\in\mathbb T: \exists\, o\in K\subset \mathbf T
            \text{ connected such that~ } s\leq |K|<\infty,\,\frac{|\partial K|}{|K|}\leq q\right\}
	\end{equation}
    has exponentially small probability
	\begin{equation}
	    \mathbb P_\mu (\AEevent_{s}) \leq \exp(-\cAE s).
	\end{equation}	
\end{lem}


We refer to \cite[Lemma 2.5]{MullerStern24} for the proof of the lemma. It is obvious that if 
Lemma~\ref{lem:ms24_lem25} holds for some $q$ then it also holds for all $0<q' \le q$ with
the same constant $\cAE$. Since we fix the constant $q$ in the next section, 
we refrain from displaying the dependence of the event $\AEevent_s$ on $q$.

\section{Proof of Theorem \ref{thm:returnprobability}}\label{sec:proof}
		
%
Throughout this section, we fix a constant 
$q \in {}]0, \min(1,\mathbf i_{\mkern 2mu\mathbb T})[\,$,
which guarantees the validity of  Lemma \ref{lem:ms24_lem25}.

\begin{defn}
		We fix $\T\in\mathbb{T}$.
    \begin{longlist}
    \item
    		The first hitting time of a vertex subset $S \subset \T$ by the simple random walk on $\T$ is given by 
        \begin{equation}
        		\tau_S\coloneqq\tau_{S}^\T \coloneqq \inf\big\{j\in\N_0: X_j\in  S  \big\}
						\in\N_0\cup\{\infty\},
				\end{equation} 
				and we write $\tau_{x} := \tau_{\{x\}}$ in case of a single vertex $x\in\T$. We warn the reader about the difference to the definition of the first hitting time \emph{after zero} $\tau_{\mathrm{oc}}$ of the $q$-oceans in \eqref{eq:tau-oc-def}.
    \item 
        For $t\in\N$ we set $q_t\coloneqq qt^{-1/3}$. The bad event
        \begin{equation}
                A_t \coloneqq A_t^{\T} \coloneqq \{\tau_{\mathcal A_{q_t}}\leq t\} \subseteq \T^{\N_{0}}
        \end{equation}
        contains those random-walk trajectories on $\T$ which visit a $q_t$-island 
        before time $t$.
     \end{longlist}
\end{defn}

We show first that the bad event $A_{t}$ is sufficiently rare under the annealed measure. 		
		
\begin{lem}
    \label{lem:A_t}
	There exists constants $c'>0$ and $t_{0}\in \N$ such that 
    \begin{equation}
        \E_{\mu}\big[P_o^{\mathbf T}(A_t)\big] \le \exp\big(- c' t^{1/3}\big)
    \end{equation}
    for every $t\in\N$ with $t\ge t_{0}$.
\end{lem}

\begin{proof}
\emph{Setup.}
The argument crucially relies on the joint consideration of the random tree and the random walk as inherent under the annealed measure.
Let $\Omega \coloneqq \bigcup_{\T\in\mathbb T} ( \{\T\} \times \T^{\N_0} )$ be the bundle consisting of the spaces of random-walk trajectories as fibre spaces over the corresponding BGW-trees as base points. 
The space $\Omega$ carries a natural trace $\sigma$-algebra, as we briefly indicate: The label space $\mathcal X$ comes with the discrete (power-set) $\sigma$-algebra $\mathcal P(\mathcal X)$ and the product space ${\mathcal X}^{\N_0}$ with the usual product $\sigma$-algebra $\bigotimes_{\N_0} \mathcal P(\mathcal X)$ thereof, generated by cylinder sets. Restricting  the product $\sigma$-algebra 
$\mathcal F \otimes \big(\bigotimes_{\N_0} \mathcal P(\mathcal X)\big)$ on $\mathbb T \times {\mathcal X}^{\N_0}$ to $\Omega$ defines the canonical $\sigma$-algebra of measurable sets in $\Omega$. 
If $\Xi_\Omega \subseteq \Omega$ is measurable then 
\begin{enumerate}
    \item[(i)] $\Xi_\Omega \cap (\{\T\} \times \T^{\N_0}) \eqqcolon \{\T\} \times \Xi^\T$ 
defines a measurable $\T$-fibre $\Xi^\T \subseteq \T^{\N_0}$ 
for every $\T\ \in\mathbb T$ and
\item[(ii)] the function 
$\mathbb T \ni \T \mapsto P_o^\T(\Xi^\T)$ is $\mathcal F$-measurable.
\end{enumerate}

Conversely, given a family of events $\{\Xi^\T\}_{\T\in\mathbb T}$ defined on the fibres $\Xi^\T \subseteq \T^{\N_0}$, then we lift it to $\Omega$   
according to
\begin{equation}
    \label{eq:events-Omega}
	\Xi_{\Omega} \coloneqq \big(\{\Xi^\T \}_{\T\in\mathbb T}\big)_\Omega 
    \coloneqq \bigcup_{\T\in\mathbb T} \big( \{\T\} \times \Xi^\T \big).
\end{equation}
(Of course, further joint measurability assumptions of $\{\Xi^\T\}_{\T\in\mathbb T}$ are required to guarantee that $\Xi_\Omega$ is an event in $\Omega$. We note that this will be the case in all applications of \eqref{eq:events-Omega} below.)
We define the annealed probability measure of the simple random walk for a measurable subset $\Xi_\Omega \subseteq \Omega$
by $Q_\mu (\Xi_\Omega)\coloneqq \E_\mu[P_o^\T(\Xi^\T)]$. The expectation with respect to $Q_\mu$ is denoted by $E_{Q_\mu}$. A pure BGW-event $F\in\mathcal F$ also has a canonical lift to $\Omega$ in the spirit of \eqref{eq:events-Omega} by setting 
\begin{equation}
    \label{eq:Fevents-Omega}
    F_\Omega \coloneqq \big(\{F^\T \}_{\T\in\mathbb T}\big)_\Omega  
    = \bigcup_{\T\in F} \big( \{\T\} \times \T^{\N_0} \big),
\end{equation}
 where $F^\T \coloneqq \emptyset$ if $\T\not\in F$ and $F^\T \coloneqq \T^{\N_0}$ if $\T\in F$.

\emph{Reduction of the main event.} 
    Now, in the event $A_{t}^{\T}$ for $\T\in\mathbb{T}$ fixed, there exists a unique first time $s\leq t$ at which the random walk hits
	a connected vertex subset $S \subset \T$ which is $\P_\mu$-almost surely finite and satisfies $|S| \geq t^{1/3}$ and $|\partial S| / |S| <q_t\leq q$ by Lemma~\ref{lem:ms24_cor29} and Remark~\ref{delta_q-pos}. Let $x_{S}$ be the vertex of $S$ which is closest to the root. Since the random walk starts at the root and $\T$ is a tree, we have $\tau_{x_{S}}^\T = s$.
This motivates the definition of the event
	\begin{equation}
		B_{s}^{\T} \coloneqq \Big\{ \,\exists\, S \subset \T \text{ finite and connected } : \tau_{x_S}^\T=s, 
            \; |S|  \ge t^{1/3}, \; \mathrm{iso}^\T_{x_{S}}(S) \le q \Big\} \subseteq \T^{\N_{0}}
	\end{equation}
for $s\in\N_{0}$ and $\T\in\mathbb T$,	where $\mathrm{iso}^\T_{x_{S}}(S)\coloneqq\frac{|\partial_\T S| - \1_{\T \setminus \{o\}}(x_S)}{|S|}$ denotes a corrected isoperimetric ratio of the vertex subset $S$ in $\mathbf T$ in the sense that $\1_{\T \setminus \{o\}}(x_{S})$ disregards one edge of the boundary (the one to the parent of $x_S$) in case $x_{S}$ is not the root.  
We lift the events $B_{s}^{\T}$ to $\Omega$ in the spirit of \eqref{eq:events-Omega} and write
$\what B_{s} \coloneqq (\{B_s^\T\}_{\T\in\mathbb T}\} )_\Omega$ for the corresponding event in $\Omega$. 
We conclude that   	
	\begin{equation}
		\E_{\mu}[P_o^{\mathbf T}(A_t^{\T})] 
		\le \sum_{s=0}^{t} \E_{\mu}[P_o^{\mathbf T}(B_s^{\T})] 
		= \sum_{s=0}^{t} Q_{\mu}(\what B_{s}) 
		= \sum_{s=0}^{t} E_{Q_{\mu}}\big( Q_{\mu}(\what B_{s} | \what X_{s}) \big), 
        \label{eq:conditionalexpectation}
	\end{equation}
	where $\what X_{s}: \Omega \rightarrow \mathcal{X} \times \{S,E\}$ is the map that assigns to any pair consisting of a tree $\T \in\mathbb T$ and of a random-walk trajectory $\omega^{\T} \in\T^{\N_{0}}$ on $\T$ the vertex position $\omega^{\T}_{s}$ (in Ulam--Harris notation) of the trajectory at time $s$ together with the type of the vertex $\omega^{\T}_{s}$ in the tree $\T$. 
    
\emph{Estimating conditional probabilities.}
In view of \eqref{eq:conditionalexpectation}, our goal is to bound the supremum of
\begin{equation}
\label{eq:cond-exp-eval}
	Q_{\mu}\big(\what B_{s} \mid \what X_{s}=(y,f)\big) 
    = Q_{\mu}\big(\what B_{s} \; \big| \; \{y \in \T_{f}\}_\Omega \cap \{X_s =y\}_\Omega\big),  
\end{equation}
over those $(y,f)\in\mathcal X\times \{S,E\}$ for which $Q_\mu\big(\what X_{s}=(y,f)\big) >0$.
Here $\{y \in \T_{f}\}_\Omega$  is defined in the sense of \eqref{eq:Fevents-Omega}, and $\{X_s =y\}_\Omega$ is a shorthand for $(\{X_s^\T =y\}_{\T\in\mathbb T})_\Omega$ in the sense of \eqref{eq:events-Omega}. 
In this last expression we have deliberately written $X_s^\T$ to emphasise the dependence of the random walk on the tree $\T$. 
Likewise, we use $\{\tau_y =s\}_\Omega$ as a shorthand for $(\{\tau_y^\T =s\}_{\T\in\mathbb T})_\Omega$ below.

Let $E_{1}, E_{2}, E_{3}, D_{1}, D_{2}$ be any events in $\Omega$ such that 
$E_{1}\cap E_{2} = D_{1} \cap D_{2}$ and $Q_\mu(E_2\cap E_3) >0$. We recall the elementary identity for the conditional probability  
\begin{equation}
	Q_{\mu}(E_{1} \mid E_{2} \cap E_{3}) 
	= \frac{Q_{\mu}(E_{1} \cap E_{2} \cap E_{3})}{Q_{\mu}(E_{2} \cap E_{3})}
	= \frac{Q_{\mu}(D_{1} \cap D_{2} \mid E_{3})}{Q_{\mu}(E_{2} \mid E_{3})}
\end{equation}
and apply it with 
\begin{equation}
    E_1\coloneqq \what B_{s}, \; E_2\coloneqq\{X_s=y\}_{\Omega}, \;
    E_{3}\coloneqq\{y\in \T_{f}\}_{\Omega}, \;
    D_1\coloneqq\{\tau_y =s\}_\Omega \text{~~and~~} D_2\coloneqq (F_y^+)_{\Omega},
\end{equation}
where
\begin{equation}
	\begin{split}
		F_{y}^{+} 
        &\coloneqq  \Big\{ \T\in\mathbb T: \,\exists\, S \subseteq \prog_{\T}(y) 
	 				\text{ finite and connected with} \\[-1ex]
		&\hspace{5cm}	y \in S, \; |S|  \ge t^{1/3}, \; \mathrm{iso}^\T_{y}(S) \le q \Big\} 
					\in\mathcal F.
	\end{split}
\end{equation}
Here, $\prog_{\T}(y)$ denotes the progeny of a vertex $y\in\T$, that is, the subtree of $\T$ emanating from $y$ as its root.
This yields
\begin{equation}
	\label{eq:cond-prob-apply}
	Q_{\mu} \big(\what B_{s} \; \big| \; \{y \in \T_{f}\}_\Omega \cap \{X_s =y\}_\Omega\big)
	= \frac{Q_{\mu}\big(\{\tau_{y}=s\}_{\Omega} \cap (F_y^+)_\Omega \mid\{y\in \T_{f}\}_\Omega\big)}{Q_{\mu}(\{X_s=y\}_{\Omega} \mid \{y\in \T_{f}\}_\Omega)} .
\end{equation}
We notice that $(\{y\in\T_f\}_\Omega)^\T=\T^{\N_0}$ if $y\in\T_f$, otherwise it is the empty set. Therefore
 \begin{equation}
     Q_\mu\big(\{y\in \T_{f}\}_\Omega\big)=\E_\mu\Big[P_o^\T\Big(\big(\{y\in \T_{f}\}_\Omega\big)^\T\Big)\Big]=\E_\mu \big[\1_{y\in \T_{f}}\cdot P_o^\T\big(\T^{\N_0}\big)\big]=\P_\mu\big[y\in \T_{f}\big].
     \label{eq:label1}
 \end{equation}
We abbreviate the event $H_\Omega  \coloneqq \{\tau_{y}=s\}_{\Omega} \cap (F_y^+)_\Omega \cap \{y\in \T_{f}\}_\Omega$. Similar to before, if $\T\in F_y^+$ and $y\in \T_{f}$, then $(H_\Omega)^\T=\{\tau_{y}^\T=s\}$ and otherwise $H^\T=\emptyset$, hence
\begin{equation}
    \label{eq:label2}
    Q_\mu\big(H_\Omega\big)
    =\E_\mu\big[P_o^\T\big(H_\Omega^\T\big)\big]=\E_\mu\big[\1_{F_y^+}(\T) \1_{y\in \T_{f}} P_o^\T\big(\tau_{y}^\T=s\big)\big].
\end{equation}
We rewrite the conditional probability in the numerator of the right-hand side of \eqref{eq:cond-prob-apply} with \eqref{eq:label1} and \eqref{eq:label2} as 
\begin{equation}
	\label{eq:pre-Markov-apply}
	\begin{split}
		Q_{\mu}\big(\{\tau_{y}=s\}_{\Omega} \cap (F_y^+)_\Omega \mid \{y\in \T_{f}\}_\Omega\big)
        &=\frac{\E_\mu\big[\1_{F_y^+}(\T) \1_{y\in \T_{f}} P_o^\T\big(\tau_{y}^\T=s\big)\big]}{\P_\mu\big(y\in \T_{f}\big)}\\
        &= \E_{\mu}\Big[ \1_{F_{y}^{+}}(\T) P_{o}^{\T}\big(\tau_y^\T =s \big) \;\Big|\; y\in \T_{f} \Big].
	\end{split}
\end{equation}
Now $F_{y}^{+}$ is exactly the kind of event with the same symbol occurring in \eqref{eq:Markov-star} and \eqref{eq:Markov}, and the map 
$\mathbb T \ni \T \mapsto P_{o}^{\T}(\tau_{y}^\T=s)$ is measurable with respect to the sigma algebra generated by all events of the kind $F_{y}^{-}$ in \eqref{eq:Markov-star} and \eqref{eq:Markov}.
Hence, the probability on the right-hand side of \eqref{eq:pre-Markov-apply} factorises due to the tree Markov property \eqref{eq:Markov}, and we get
\begin{equation}
	\label{eq:Markov-apply}
	Q_{\mu}\big(\{\tau_{y}=s\}_{\Omega} \cap (F_y^+)_\Omega \mid \{y\in \T_{f}\}_\Omega\big)
	= \P_{\mu}(F_{y}^{+} \mid y\in\T_{f}) \; 
        \E_{\mu} \big[P_{o}^{\T}(\tau_y^\T =s ) \;\big|\; y\in\T_{f}\big].
\end{equation}
Plugging \eqref{eq:Markov-apply} into \eqref{eq:cond-prob-apply} and observing the identity
$Q_{\mu}\big(\{X_s =y\}_{\Omega} \mid \{y\in \T_{f}\}_\Omega\big)
= \E_{\mu}\big[ P_{o}^{\T}(X_s^\T =y) \;\big|\; y\in \T_{f} \big]$, which follows in analogy to \eqref{eq:pre-Markov-apply}, we infer 
\begin{equation}
	\label{eq:cond-exp2}
	\begin{split}
		Q_{\mu} \big(\what B_{s} \;\big|\; \{y \in \T_{f}\}_\Omega \cap \{X_s =y\}_\Omega\big) 
		&=	\frac{ \E_{\mu} \big[P_{o}^{\T}(\tau_y^\T =s ) \;\big|\; y\in\T_{f}\big]}{\E_{\mu}\big[ P_{o}^{\T}(X_s^\T=y ) \;\big|\; y\in \T_{f} \big]}
			\; \P_{\mu}(F_{y}^{+} \mid y\in\T_{f}) \\
		&\le \P_{\mu}(F_{y}^{+} \mid y\in\T_{f}).
	\end{split}
\end{equation}
Here, we used the inclusion $\{\tau_{y}^\T=s\}  \subseteq \{X_s^\T=y\}$ for every $\T\in\mathbb T$ to obtain the inequality. 

It remains to distinguish the two cases $f\in \{S,E\}$ for the type of $y$ in order to bound the right-hand side of \eqref{eq:cond-exp2}. If $f=S$, that is, $y$ is a survivor, \eqref{eq:shift-S} and Lemma \ref{lem:ms24_lem25}  yield 
\begin{equation}
	\label{eq:B_surv_1}
	\P_{\mu}(F_{y}^{+} \mid y\in\T_{S}) 
	= \P_{\mu}(F_{o}^{+}) 
	= \P_\mu(\AEevent_{\lceil t^{1/3}\rceil})
	\le \exp(-c_{\AEevent} t^{1/3}),
\end{equation}
where $\lceil r \rceil$ is the smallest integer that is not smaller than $r\in\R$.
On the other hand, if $y$ goes extinct, that is, $f=E$, we conclude from \eqref{eq:shift-E} and 
\eqref{eq:finitetreetails} that
\begin{equation}
	\begin{split}
		\label{eq:xi_extinct_root_0}
  	\P_{\mu}(F_{y}^{+} \mid y\in\T_{E}) 
  	&= \P_\mu^{*}(F_{o}^{+} \mid |\T| < \infty) \\
  	&\le \P_\mu^{*}( |\T| \ge t^{1/3} \mid |\T| < \infty)
			\le \exp\big(-c_{\mathrm{fin}}\,t^{1/3}\big).
	\end{split}
\end{equation}
Thus, we deduce from \eqref{eq:xi_extinct_root_0}, \eqref{eq:B_surv_1}, \eqref{eq:cond-exp2} and 
\eqref{eq:cond-exp-eval} that for all $(y,f) \in \mathcal X \times \{S,E\}$ with 
$Q_\mu\big(\what X_s =(y,f)\big) >0$ we have the estimate
\begin{equation}
	Q_{\mu}\big(\what B_{s} \mid \what X_{s}=(y,f)\big) 
	\le \exp\big(-c'' t^{1/3}\big)
\end{equation}
with $c'' \coloneqq \min\{c_{\mathrm{fin}}, \cAE\} >0$. Therefore, the same upper bound holds $Q_\mu$-almost surely for the random variable $Q_{\mu}(\what B_{s} | \what X_{s})$. Now, the claim follows from 
\eqref{eq:conditionalexpectation}.
\end{proof}

The event of returning to the root in $t$ steps without visiting any large $q$-islands can 
be dealt with for each $\T$. We recall that
$q_t = qt^{-1/3} \leq q$.
        
\begin{lem}
    \label{lem:complement}
	For every rooted infinite tree $\mathbf T$ which is locally finite, every $x\in\mathcal X$ and every $t\in\N$, 
    we have 
    \begin{equation}
        P_o^{\mathbf T}\big(\{X_t=x\}\cap  A_t^\complement\big) \leq \sqrt{\frac{\mathrm{deg}_T(x)}{\mathrm{deg}_T(o)}}\cdot \exp\big(- q^2t^{1/3}/18\big).
    \end{equation}
\end{lem}
		
\begin{proof}		
Without loss of generality, we assume that $x\in\T\setminus A_{q_t}$, otherwise the probability on the left hand side is $0$.
    By definition of the event $A_t^\complement$, we have
    \begin{equation}
        \label{eq:reduce-to-W}
        P_o^{\mathbf T}\big(\{X_t=x\}\cap  A_t^\complement\big) 
        = P_o^{\mathbf T}\big(X_t=x,X_s\notin \mathcal A_{q_t}\, 
            \forall \, s\in \{0, \ldots, t\} \big)
	    \le P_o^{{\mathbf T}_{q_t}}(W_t=x).    
    \end{equation}
    Here, the second inequality holds because the simple random walk 
    $(X_s)_{s\in\{0,\ldots, t\}}$ jumps only between vertices in 
    $\mathbf T \setminus \mathcal A_{q_t}$, and -- since $\mathbf T$ is a tree and the walk does not enter $q_t$-islands -- no two 
    consecutive vertices $y_1,y_2$ on its path belong to the outer vertex boundary of the same 
    $q_t$-island of $\mathbf T$. Hence, we conclude from Remark~\ref{rem:wq-prop}\kern1pt(i) that $w_{q_t}(y_1,y_2) = 1$ 
    for any two such consecutive vertices, and the inequality arises because the 
    requirements that $(W_s)_{s\in\{0,\ldots, t\}}$ must not jump over $q_t$-islands or 
    stay at a vertex have been dropped in the rightmost probability in \eqref{eq:reduce-to-W}.

    Finally, we rewrite the return probability on the right-hand side of 
    \eqref{eq:reduce-to-W} in terms of the corresponding Markov operator and apply 
    Theorem \ref{thm:markovkernelbound}
	\begin{equation}
		\begin{split}
       P_o^{{\mathbf T}_{q_t}}(W_t=x)
			&= \frac{1}{w_{q_t}(o)}\langle\delta_{o},\mathbf P_{{q_t}}^{\mkern3mu t}\delta_{x} 
            \rangle_{\mathbf T_{q_t},w_{q_t}} \leq \sqrt{\frac{\mathrm{deg}_T(x)}{\mathrm{deg}_T(o)}}\cdot\|\mathbf P_{{q_t}}\|_{\mathbf T_{q_t},w_{q_t}}^{\mkern3mu t}\\
        &\leq \sqrt{\frac{\mathrm{deg}_T(x)}{\mathrm{deg}_T(o)}}\cdot \left(1-\frac{q_t^2}{18}\right)^{t}\leq\sqrt{\frac{\mathrm{deg}_T(x)}{\mathrm{deg}_T(o)}}\cdot \exp\left(-\frac{q^2t^{1/3}}{18}\right).  
        \\[-3ex]
	\end{split}
	\end{equation}    
\end{proof}

The proof of the main theorem now follows as a direct consequence of the previously assembled results.
		
\begin{proof}[Proof of Theorem \ref{thm:returnprobability}]
    The claim is trivially true for $t=0$, so we assume $t\in\N$ from now on. 
    We recall that $q$ is fixed such that Lemma \ref{lem:ms24_lem25} holds. 
    For the time being, we fix $t\in\mathbb N$ such that $t\geq t_0$ and infer from 
    Lemma~\ref{lem:complement} and Lemma~\ref{lem:A_t} that  
	\begin{equation}
	    \begin{split}
		      \E_{\mu}\big[P_o^{\mathbf T}(X_t=x)\big] 
            &= \E_{\mu}\big[P_o^{\mathbf T}(\{X_t=x\}\cap A_t^\complement)\big] 
                + \E_{\mu}\big[P_o^{\mathbf T}(A_t)\big] \\
			&\le \E_\mu\left[\sqrt{\frac{\mathrm{deg}_T(x)}{\mathrm{deg}_T(o)}} \,\right] \exp\big(- q^2t^{1/3}/18\big) +  \exp\big(- c' t^{1/3}\big).
			\end{split}
	\end{equation}
    We note that $\E_\mu\Big[\sqrt{\frac{\mathrm{deg}_T(x)}{\mathrm{deg}_T(o)}}\,\Big]\leq \E_\mu\left[\mathrm{deg}_T(x)\right] < \infty$ because the first moment of the offspring distribution $\mu$ exists. 
    Therefore, there are $t_0'\geq t_{0}$ and $c''>0$, such that for all $t > t_0'$
    \begin{equation}
      \E_{\mu}\big[P_o^{\mathbf T}(X_t=x)\big] \leq \exp\big(-c''  t^{1/3}\big).\label{eq:greatert_0}
    \end{equation}
    We then choose $c_0>0$ sufficiently small such that
    \begin{equation}
    \E_{\mu}\big[P_o^{\mathbf T}(X_t=x)\big] \leq \exp\big(-c_0 ( t_0')^{1/3}\big)\leq \exp\big(-c_0 t^{1/3}\big)   \label{eq:leqt_0}     
    \end{equation}
    for every $t\in\N$ with $t\leq t_0'$.
    This is possible because $\E_{\mu}[P_o^{\mathbf T}(X_t=x)] < 1$ for all such $t$.
    Together, \eqref{eq:greatert_0} and \eqref{eq:leqt_0} yield the theorem with 
    $c\coloneqq\min(c_0,c'')$.
\end{proof}

\section{Proof of Theorem~\ref{thm:lifshits}}
\label{sec:lifshits}



The following lemma will be applied with $\mathbf{G}$ as the giant cluster. It estimates the relevant spectral property of the graph Laplacian in terms of the one of the normalised graph Laplacian.

\begin{lem}
    \label{lem:Delta-DeltaSchlange}
	Let $\mathbf{G}$ be a \emph{finite} connected graph and 
    let $\Delta_{\mathbf{G}}$, respectively 
	$\wtilde\Delta_{\mathbf{G}}$, be 
	the associated graph Laplacian, respectively normalised graph Laplacian. For $E>0$ let $I_{E}$ be any interval whose interior is given by $\mathopen]0,E\mathclose[\,$. Then we have 
	\begin{equation}
		\tr_{\ell^{2}(\G)} \1_{I_{E}}(\Delta_{\mathbf{G}}) 
		\le \tr_{\ell^{2}(\G)} \1_{I_{E}}(\wtilde\Delta_{\mathbf{G}}).
	\end{equation}
\end{lem}

\begin{proof}
	If $|\G|=1$, then $\Delta_{\mathbf{G}} = \wtilde\Delta_{\mathbf{G}}=0$, and the claim is trivial. 
	So let us assume $|\G| \ge 2$ for the rest of this proof. Since $\mathbf{G}$ is finite and connected 
	we infer $\deg_{\mathbf{G}}(v) \in\mathbb N$ for every $v\in V$. We introduce the matrix 
	$a\coloneqq \Delta_{\mathbf{G}}^{1/2} D_{\mathbf{G}}^{-1/2}$ and conclude $0 \le aa^{*} \le \Delta_{\mathbf{G}}$ 
	in the sense of quadric forms. The Min-Max-Theorem leads to component-wise inequalities \cite[Cor.\ III.2.3]{Bhatia97} between the two decreasingly ordered sequences of eigenvalues of $aa^{*}$ and $\Delta_{\mathbf{G}}$. In turn, these eigenvalue inequalities immediately imply
	\begin{equation}
	    \label{eq:delta-aastern}
		\tr_{\ell^{2}(\G)} \1_{I_{E}}(\Delta_{\mathbf{G}}) 
		\le \tr_{\ell^{2}(\G)} \1_{I_{E}}(aa^{*})
	\end{equation}
	if $0 \in I_{E}$. But \eqref{eq:delta-aastern} also holds if $0 \not\in I_{E}$ because 
	the one-dimensional null space $\mathcal{N}$ of $\Delta_{\mathbf{G}}$ coincides with the null space 
	$\what{\mathcal{N}}$ of $aa^{*}$, as we shall argue now: 
    We denote the edge set of $\G$ by $\mathcal E_\G$.
	Since $\|\Delta_{\mathbf{G}}^{1/2}\psi\|^{2} = \langle\psi,\Delta_{\mathbf{G}}\psi\rangle
		= \sum_{\{x,y\} \in \mathcal E_\G} |\psi(x) - \psi(y)|^{2}$ for every $\psi\in\ell^{2}(\G)$, 
	we note first that the null space of $\Delta_{\mathbf{G}}^{1/2}$ coincides with $\mathcal{N}$, whence 
	$\mathcal{N} \subseteq \what{\mathcal{N}}$. In order to show the converse inclusion
	$\what{\mathcal{N}} \subseteq \mathcal{N}$, we pick an arbitrary vector $\psi\in \what{\mathcal{N}}$.
	Hence, $0= \langle \psi, aa^{*}\psi\rangle = \sum_{x\in \G} |\phi(x)|^{2}/\deg_{\mathbf{G}}(x)$, 
	where $\phi \coloneqq\Delta_{\mathbf{G}}^{1/2} \psi$. But this implies $\phi=0$, whence 
	$\psi\in \mathcal{N}$, and \eqref{eq:delta-aastern} is established.
	
	Finally, the characteristic polynomials of the matrices $aa^{*}$ and $a^{*}a = \wtilde{\Delta}_{\mathbf{G}}$ coincide \cite[Thm.\ 1.3.22]{HoJo13} so that $aa^{*}$ and $\wtilde{\Delta}_{\mathbf{G}}$ have the same eigenvalues including multiplicities.  We conclude that  
	\begin{equation}
		\label{eq:aastern-asterna}
		\tr_{\ell^{2}(\G)} \1_{I_{E}}(aa^{*}) 
		= \tr_{\ell^{2}(\G)} \1_{I_{E}}(\wtilde\Delta_{{\mathbf {G}}}),
	\end{equation}
	and the claim follows from \eqref{eq:delta-aastern} and \eqref{eq:aastern-asterna}.    
\end{proof}

The next lemma estimates the average, conditioned on non-extinction, of the low-energy spectral projection of the normalised Laplacian in terms of the annealed return probability of the discrete-time random walk from Theorem~\ref{thm:returnprobability}. It will be applied with $\mu = \mathrm{Poi}(\lambda)$ in the proof of
Theorem~\ref{thm:lifshits}.

\begin{lem}\label{lem:return-specproj}	
	Assume that the first moment $\lambda$ of the offspring distribution $\mu$ exists 
    and satisfies  $\lambda>1$.	Then, there exist constants $E_0 \in \,]0,1[\,$ and $\alpha >0$ such that 
	\begin{equation}  \label{eq:specproj}
		\E_\mu\big[\langle\delta_o, \1_{[0,E]}(\wtilde\Delta_\T) \delta_o\rangle\big] 
        \le \exp\big(-\alpha E^{-1/2}\big)
	\end{equation}
	for every $E\in \,]0,E_0]$.
\end{lem} 


\begin{proof}
    The generator of the discrete-time random walk $\{X_t\}_{t\in\N_0}$ on a locally finite connected 
    graph $\G$ with a countable vertex set (and at least 2 vertices) is determined by the transition probability 
    $P_o^\G(X_{t+1} =y \mid X_t =x) = (D_\G^{-1}A_\G )_{xy}$ for every $x,y \in\G$ and every $t\in\N_0$ so that the return probability on $\G$ is given by 
    \begin{equation}
        P_o^\G(X_t =o) = \langle\delta_o, (\id - \widetilde\Delta_\G)^t\delta_o \rangle
    \end{equation}
    for every $t\in\N_0$. Thus, we infer from Theorem~\ref{thm:returnprobability} that
    \begin{equation}
        \E_\mu \big[ \langle\delta_o, (\id - \widetilde\Delta_\T)^t\delta_o \rangle \big]
        \le \exp\big(-ct^{1/3}\big).
    \end{equation}
    We assume without loss of generality that the constant satisfies $c<1$. Otherwise replace $c$ by some positive number $c'$ such that $c'<1\leq c$ in the rest of this proof. Because of $c<1$ there exists a constant $E_0 \coloneqq E_0(c) \in \,]0,1[\,$ such that $\ln (1-E) \ge - E/c$ for every $E \in [0, E_0]$. This and the estimate $\1_{[0,E]}(\lambda) \le (1-E)^{-t} (1- \lambda)^t$ for every $\lambda \in [0, 1]$ and 
    $E \in \,]0,1[\,$ imply
    \begin{equation}
        \label{eq:spec-proj-thm-apply}
        \E_\mu\big[\langle\delta_o, \1_{[0,E]}(\wtilde\Delta_\T) \delta_o\rangle\big]
        \le \exp\big\{- c (t^{1/3} - tE/c^2) \big\}
    \end{equation}
    for every $E \in [0, E_0]$ and every $t\in\N_0$. Now, we choose $t = \lfloor (2E/c^2)^{-3/2}\rfloor$,
    where $\lfloor r\rfloor$ denotes the largest integer not exceeding $r\in\R$, and 
    observe $(x-1)^{1/3} > x^{1/3} -1$ for any $x>1$
    so that the exponent on the right-hand side of \eqref{eq:spec-proj-thm-apply} is bounded from above by 
    \begin{equation}
        -c \Big[ (2E/c^2)^{-1/2} - 1 - (2E/c^2)^{-3/2} E/c^2\Big] 
        \le  -(c/2)^2 (\sqrt{2} -1) E^{-1/2}.
    \end{equation}
    For the last inequality to hold, we assume that $E_0 \le c^2/16$, whence $1 \le (16E/c^2)^{-1/2} $. The claim follows.  
\end{proof}

Whereas the preceding lemma was concerned with the BGWT-average conditioned on non-extinction for the normalised Laplacian that relies on Theorem~\ref{thm:returnprobability} and holds for general offspring distributions $\mu$, the next lemma treats the BGWT-average conditioned on extinction for the Laplacian, which is much easier and follows from the considerations in \cite{KiKhMu06}. Morally, the behaviour described in the next lemma is due to finite, linear trees. 
We write $\Lambda_\lambda \coloneqq \Lambda(\mathrm{Poi}(\lambda))$ for the extinction probability 
\eqref{eq:ext-prob-def} of the BGWT with offspring distribution $\mu= \mathrm{Poi}(\lambda)$ and recall that $0 < \lambda\Lambda_\lambda <1$ for every $\lambda >1$.

\begin{lem}
    \label{lem:finite-trees}
    For all $\lambda >1$ and $E>0$ we have 
    \begin{equation}
        \frac{\e^{-f^-_\lambda }}{2\lambda\Lambda_\lambda}
        \e^{-2\sqrt{3}f^-_\lambda   E^{-1/2}}
        \le \E^*_{\mathrm{Poi}(\lambda)} \big[\langle\delta_o, \1_{]0,E]}(\Delta_\T) \delta_o\rangle 
            \;\big|\; |\T| < \infty \big] 
        \le \frac{\e^{f^+_\lambda}}{\lambda\Lambda_\lambda}
            \e^{-f^+_\lambda E^{-1/2}},
    \end{equation}
    where $f^\pm_\lambda \coloneqq  \lambda\Lambda_\lambda - \ln (\lambda\Lambda_\lambda) - (1 \pm 1)/2 >0$. 
\end{lem}

\begin{proof}
    Let $\lambda >1$. For $\mu = \mathrm{Poi}(\lambda)$ the offspring distribution 
    \eqref{eq:E-multitype_offspring_distributions} of vertices going extinct is given by        
    $\mu_E = \mathrm{Poi}(\lambda\Lambda_\lambda)$, where $\lambda\Lambda_\lambda<1$. Hence, it follows from the right equality in
    \eqref{eq:shift-E} and \eqref{BJ-limit} that 
    \begin{equation}
        \E^*_{\mathrm{Poi}(\lambda)} \big[\langle\delta_o, \1_{]0,E]}(\Delta_\T) \delta_o\rangle 
            \;\big|\; |\T| < \infty \big] = \sigma_{\lambda\Lambda_\lambda}(\,]0,E]).       
    \end{equation}
    Thus, the claim follows from Lemmas~3.1 and~4.1 in \cite{KiKhMu06}. We note that an inspection of the proof of Lemma~3.1 in \cite{KiKhMu06} shows that the constant $C$ in its claim can be chosen as $1$. 
\end{proof}

We are now able to conduct the 

\begin{proof}[Proof of Theorem~\ref{thm:lifshits}]
    We fix $\lambda >1$ and a spectral value $E>0$. A decomposition with respect to extinction and survival of the representation 
    \eqref{BJ-limit} for $\sigma_\lambda$ gives
    \begin{equation}
        \label{eq:BJ-apply}
        \begin{split}
            \sigma_\lambda(\,]0,E]) 
            &= \Lambda_\lambda \; \E^*_{\mathrm{Poi}(\lambda)} \big[\langle \delta_{o}, 
                \1_{]0,E]}(\Delta_{\T}) \delta_{o}\rangle \;\big|\; |\T| < \infty\big]  \\
            &\quad + (1 - \Lambda_\lambda) \;\E_{\mathrm{Poi}(\lambda)} \big[\langle \delta_{o}, 
                \1_{]0,E]}(\Delta_{\T}) \delta_{o}\rangle \big].    
        \end{split}
    \end{equation}
    \emph{Lower bound.} \quad 
    After dropping the (non-negative) second term on the right-hand side in \eqref{eq:BJ-apply}, 
    the desired lower bound is the one in Lemma~\ref{lem:finite-trees}. \\
    \emph{Upper bound.} \quad 
    The first term on the right-hand side in \eqref{eq:BJ-apply} is appropriately bounded from above 
    by Lemma~\ref{lem:finite-trees}. It remains to find an analogous upper bound for the expectation in the 
    second line of \eqref{eq:BJ-apply}: We choose a monotone decreasing sequence $(E_n)_{n\in\N} \subset [E, \infty[\,$, which converges to $E$ from above and such that each $E_n$ is neither an atom of the probability measure $\E_{\mathrm{Poi}(\lambda)} \big[\langle \delta_{o}, \1_{\bullet}(\Delta_{\T}) \delta_{o}\rangle \big]$ nor of $\E_{\mathrm{Poi}(\lambda)} \big[\langle \delta_{o}, \1_{\bullet}(\wtilde\Delta_{\T}) \delta_{o}\rangle \big]$. This is possible because both sets of atoms are countable. We also remark that the spectral value $0$ is not an atom of any of the two because $\mathbf T$ is infinite $\P_{\mathrm{Poi}(\lambda)}$-almost surely. Therefore the expectation in the second line of \eqref{eq:BJ-apply} is bounded from above 
    by 
    \begin{equation}
        \label{eq:giant-weak-local}
        \E_{\mathrm{Poi}(\lambda)} \big[\langle \delta_{o}, 
            \1_{]0,E_n]}(\Delta_{\T}) \delta_{o}\rangle \big] 
        = \lim_{N\to\infty} \E_{\lambda}^{(N)} \bigg[ \frac{1}{|\mathbf C_{\mathrm{max}}^{(N)}|} 
            \tr_{\ell^2(\mathbf C_{\mathrm{max}}^{(N)})} \Big\{ \1_{]0,E_n]} (\Delta_{\mathbf C_{\mathrm{max}}^{(N)}})\Big\}\bigg].
    \end{equation}
    The equality in \eqref{eq:giant-weak-local} holds because the giant cluster 
    $\mathbf C_{\mathrm{max}}^{(N)}$ of \ER{} random graphs $\mathcal G(N,\lambda/N)$ with $\lambda>1$ 
    converges in the weak local 
    sense to the ensemble of Bienaym\'e--Galton--Watson trees conditioned on non-extinction with offspring 
    distribution $\mu=\mathrm{Poi}(\lambda)$, see Thms.\ 2.28 and 2.32 in \cite{Hofstad24}, 
    as  well as the proof of Thm.\ 2.34 therein, where the hypothesis of Thms.\ 2.28 and 2.32 is verified 
    for \ER{} graphs. Weak convergence of the spectral measures of the Laplacian as claimed in 
    \eqref{eq:giant-weak-local} then follows along the lines in the proof of \cite[Prop.\ 2.2 (ii)]{Bor17} with the adjacency operator replaced by the Laplacian. The same holds true for the normalised Laplacian 
    $\wtilde\Delta_\T$, as it is even bounded and as local as the Laplacian. This argument becomes relevant after an application of Lemma~\ref{lem:Delta-DeltaSchlange} to the right-hand side of 
    \eqref{eq:giant-weak-local} so that altogether
    \begin{equation}
        \begin{split}
            \E_{\mathrm{Poi}(\lambda)} \big[\langle \delta_{o},  \1_{]0,E]}(\Delta_{\T}) 
                \delta_{o}\rangle \big]
            &\le \lim_{n\to\infty} \lim_{N\to\infty} \E_{\lambda}^{(N)} 
                \bigg[ \frac{N}{|\mathbf C_{\mathrm{max}}^{(N)}|} \; \frac1N
                \tr_{\ell^2(\mathbf C_{\mathrm{max}}^{(N)})} \Big\{ \1_{]0,E_n]} 
                (\wtilde\Delta_{\mathbf C_{\mathrm{max}}^{(N)}})\Big\}\bigg] \\
            &= \E_{\mathrm{Poi}(\lambda)} \big[\langle \delta_{o}, 
                \1_{]0,E]}(\wtilde\Delta_{\T}) \delta_{o}\rangle \big] \\
            &\le \exp\big(-\alpha E^{-1/2}\big),
        \end{split}
    \end{equation}
    where the last step results from an application of Lemma~\ref{lem:return-specproj}. The argument is 
    complete. 
\end{proof}

%
%
%
%


	\bibliography{bibl}

@article{ChChFr86,
  author={Chayes, J. T. and Chayes, L. and Franz, J. R. and Sethna, J. P. and Trugman, S. A.},
  title={On the density of states for the quantum percolation problem},
  journal={J. Phys. A, Math. Gen.},
  fjournal={Journal of Physics A: Mathematical and General},
  volume={19},
  pages={L1173--L1177},
  year={1986},
  doi = {10.1088/0305-4470/19/18/011},
  MRNUMBER = {869300},
}

@article {Virag00,
    AUTHOR = {Vir\'ag, B.},
     TITLE = {Anchored expansion and random walk},
   JOURNAL = {Geom. Funct. Anal.},
  FJOURNAL = {Geometric and Functional Analysis},
    VOLUME = {10},
      YEAR = {2000},
    NUMBER = {6},
     PAGES = {1588--1605},
      ISSN = {1016-443X,1420-8970},
   MRCLASS = {60K99 (60J10)},
  MRNUMBER = {1810755},
MRREVIEWER = {Didier\ Piau},
       DOI = {10.1007/PL00001663},
 }

@misc {HeMuTe-inprep,
    AUTHOR = {Heydenreich, M. and M\"uller, P. and Terveer, S.},
     TITLE = {Continuous-time random walks on {B}ienaym\'e--{G}alton--{W}atson trees},
      YEAR = {2026},
      NOTE = {{I}n preparation},      
}

@article {ArBo24,
    AUTHOR = {Arras, Adam and Bordenave, Charles},
     TITLE = {Existence of absolutely continuous spectrum for
              {G}alton--{W}atson random trees},
   JOURNAL = {Commun. Math. Phys.},
  FJOURNAL = {Communications in Mathematical Physics},
    VOLUME = {403},
      YEAR = {2023},
    NUMBER = {1},
     PAGES = {495--527},
      ISSN = {0010-3616,1432-0916},
   MRCLASS = {60K35 (05C05 60J80 82B44)},
  MRNUMBER = {4645723},
MRREVIEWER = {Guilherme\ Henrique\ de Paula Reis},
       DOI = {10.1007/s00220-023-04798-3},
       URL = {https://doi-org.emedien.ub.uni-muenchen.de/10.1007/s00220-023-04798-3},
}

@article{BePaPf18,
 author = {Benjamini, Itai and Paquette, Elliot and Pfeffer, Joshua},
 title = {Anchored expansion, speed and the {Poisson}-{Voronoi} tessellation in symmetric spaces},
 fjournal = {The Annals of Probability},
 journal = {Ann. Probab.},
 issn = {0091-1798},
 volume = {46},
 number = {4},
 pages = {1917--1956},
 year = {2018},
 language = {English},
 doi = {10.1214/17-AOP1216},
 zbMATH = {6919015},
 Zbl = {1430.60019}
}

@incollection {Bor17,
    AUTHOR = {Bordenave, Charles},
     TITLE = {Spectrum of random graphs},
 BOOKTITLE = {Advanced topics in random matrices},
    SERIES = {Panor. Synth\`eses},
    VOLUME = {53},
     PAGES = {91--150},
 PUBLISHER = {Soc. Math. France, Paris},
      YEAR = {2017},
      ISBN = {978-2-85629-850-3},
   MRCLASS = {60B20 (05C80 15B52)},
  MRNUMBER = {3792625},
}

@article {BoSeVi17,
    AUTHOR = {Bordenave, Charles and Sen, Arnab and Vir\'ag, B\'alint},
     TITLE = {Mean quantum percolation},
   JOURNAL = {J. Eur. Math. Soc. (JEMS)},
  FJOURNAL = {Journal of the European Mathematical Society (JEMS)},
    VOLUME = {19},
      YEAR = {2017},
    NUMBER = {12},
     PAGES = {3679--3707},
      ISSN = {1435-9855,1435-9863},
   MRCLASS = {60K35 (05C80 60B20 60C05)},
  MRNUMBER = {3730511},
       DOI = {10.4171/JEMS/750},
}

@article {BoLeSa11,
    AUTHOR = {Bordenave, Charles and Lelarge, Marc and Salez, Justin},
     TITLE = {The rank of diluted random graphs},
   JOURNAL = {Ann. Probab.},
  FJOURNAL = {The Annals of Probability},
    VOLUME = {39},
      YEAR = {2011},
    NUMBER = {3},
     PAGES = {1097--1121},
      ISSN = {0091-1798,2168-894X},
   MRCLASS = {05C80 (05C50 15B52 47A10 60B20)},
  MRNUMBER = {2789584},
MRREVIEWER = {Marianna\ Bolla},
       DOI = {10.1214/10-AOP567},
}

@article {Salez20,
    AUTHOR = {Salez, Justin},
     TITLE = {Spectral atoms of unimodular random trees},
   JOURNAL = {J. Eur. Math. Soc. (JEMS)},
  FJOURNAL = {Journal of the European Mathematical Society (JEMS)},
    VOLUME = {22},
      YEAR = {2020},
    NUMBER = {2},
     PAGES = {345--363},
      ISSN = {1435-9855,1435-9863},
   MRCLASS = {05C80 (60B20 60C05)},
  MRNUMBER = {4049219},
MRREVIEWER = {Anant\ P.\ Godbole},
       DOI = {10.4171/jems/923},
}

@article {CosteSalez21,
    AUTHOR = {Coste, Simon and Salez, Justin},
     TITLE = {Emergence of extended states at zero in the spectrum of sparse
              random graphs},
   JOURNAL = {Ann. Probab.},
  FJOURNAL = {The Annals of Probability},
    VOLUME = {49},
      YEAR = {2021},
    NUMBER = {4},
     PAGES = {2012--2030},
      ISSN = {0091-1798,2168-894X},
   MRCLASS = {60B20 (05C80 47A10)},
  MRNUMBER = {4260473},
MRREVIEWER = {Zakhar\ Kabluchko},
       DOI = {10.1214/20-aop1499},
}

@article {MR4058984,
    AUTHOR = {Huang, Jiaoyang and Landon, Benjamin},
     TITLE = {Spectral statistics of sparse {E}rd{\H{o}}s--{R}\'enyi graph
              {L}aplacians},
   JOURNAL = {Ann. Inst. Henri Poincar\'e{} Probab. Stat.},
  FJOURNAL = {Annales de l'Institut Henri Poincar\'e{} Probabilit\'es et
              Statistiques},
    VOLUME = {56},
      YEAR = {2020},
    NUMBER = {1},
     PAGES = {120--154},
      ISSN = {0246-0203,1778-7017},
   MRCLASS = {15B52 (05C50 60B20)},
  MRNUMBER = {4058984},
MRREVIEWER = {Thomas\ Kriecherbauer},
       DOI = {10.1214/19-AIHP957},
}

@article {MR2206341,
    AUTHOR = {Bryc, W{\l}odzimierz and Dembo, Amir and Jiang, Tiefeng},
     TITLE = {Spectral measure of large random {H}ankel, {M}arkov and
              {T}oeplitz matrices},
   JOURNAL = {Ann. Probab.},
  FJOURNAL = {The Annals of Probability},
    VOLUME = {34},
      YEAR = {2006},
    NUMBER = {1},
     PAGES = {1--38},
      ISSN = {0091-1798,2168-894X},
   MRCLASS = {60F99 (15A52 60F10 62H10)},
  MRNUMBER = {2206341},
MRREVIEWER = {Martin\ V.\ Hildebrand},
       DOI = {10.1214/009117905000000495},
       URL = {https://doi-org.emedien.ub.uni-muenchen.de/10.1214/009117905000000495},
}

@article {MR2967963,
    AUTHOR = {Jiang, Tiefeng},
     TITLE = {Empirical distributions of {L}aplacian matrices of large
              dilute random graphs},
   JOURNAL = {Random Matrices Theory Appl.},
  FJOURNAL = {Random Matrices. Theory and Applications},
    VOLUME = {1},
      YEAR = {2012},
    NUMBER = {3},
     PAGES = {1250004, 20},
      ISSN = {2010-3263,2010-3271},
   MRCLASS = {05C80 (05C50 60B10 60B20)},
  MRNUMBER = {2967963},
       DOI = {10.1142/S2010326312500049},
}

@article{CampbellORourke24, 
year = {2024}, 
title = {{Spectrum of Lévy–Khintchine random Laplacian matrices}}, 
author = {Campbell, Andrew and O’Rourke, Sean}, 
fjournal = {Journal of Theoretical Probability}, 
journal = {J. Theor. Probab.}, 
issn = {0894-9840}, 
doi = {10.1007/s10959-023-01275-4}, 
pages = {933--973}, 
number = {1}, 
volume = {37}, 
}

@article {MR3945756,
    AUTHOR = {Benaych-Georges, Florent and Bordenave, Charles and Knowles,
              Antti},
     TITLE = {Largest eigenvalues of sparse inhomogeneous {E}rd{\H{o}}s-{R}\'enyi graphs},
   JOURNAL = {Ann. Probab.},
  FJOURNAL = {The Annals of Probability},
    VOLUME = {47},
      YEAR = {2019},
    NUMBER = {3},
     PAGES = {1653--1676},
      ISSN = {0091-1798,2168-894X},
   MRCLASS = {60B20 (05C80 15B52)},
  MRNUMBER = {3945756},
MRREVIEWER = {Longmin\ Wang},
       DOI = {10.1214/18-AOP1293},
}

@article {MR4288336,
    AUTHOR = {He, Yukun and Knowles, Antti},
     TITLE = {Fluctuations of extreme eigenvalues of sparse {E}rd{\H{o}}s--{R}\'enyi graphs},
   JOURNAL = {Probab. Theory Relat. Fields},
  FJOURNAL = {Probability Theory and Related Fields},
    VOLUME = {180},
      YEAR = {2021},
    NUMBER = {3-4},
     PAGES = {985--1056},
      ISSN = {0178-8051,1432-2064},
   MRCLASS = {05C80 (05C50 15B52 60B20)},
  MRNUMBER = {4288336},
MRREVIEWER = {Ke\ Wang},
       DOI = {10.1007/s00440-021-01054-4},
       URL = {https://doi-org.emedien.ub.uni-muenchen.de/10.1007/s00440-021-01054-4},
}

@article {MR4021251,
    AUTHOR = {He, Yukun and Knowles, Antti and Marcozzi, Matteo},
     TITLE = {Local law and complete eigenvector delocalization for
              supercritical {E}rd{\H{o}}s--{R}\'enyi graphs},
   JOURNAL = {Ann. Probab.},
  FJOURNAL = {The Annals of Probability},
    VOLUME = {47},
      YEAR = {2019},
    NUMBER = {5},
     PAGES = {3278--3302},
      ISSN = {0091-1798,2168-894X},
   MRCLASS = {60B20 (05C80 15B52)},
  MRNUMBER = {4021251},
MRREVIEWER = {Debankur\ Mukherjee},
       DOI = {10.1214/19-AOP1339},
       URL = {https://doi-org.emedien.ub.uni-muenchen.de/10.1214/19-AOP1339},
}

@article {MR3098073,
    AUTHOR = {Erd{\H{o}}s, L\'aszl\'o{} and Knowles, Antti and Yau, Horng-Tzer
              and Yin, Jun},
     TITLE = {Spectral statistics of {E}rd{\H{o}}s--{R}\'enyi graphs {I}:
              {L}ocal semicircle law},
   JOURNAL = {Ann. Probab.},
  FJOURNAL = {The Annals of Probability},
    VOLUME = {41},
      YEAR = {2013},
    NUMBER = {3B},
     PAGES = {2279--2375},
      ISSN = {0091-1798,2168-894X},
   MRCLASS = {60B20 (05C80)},
  MRNUMBER = {3098073},
MRREVIEWER = {Longmin\ Wang},
       DOI = {10.1214/11-AOP734},
       URL = {https://doi-org.emedien.ub.uni-muenchen.de/10.1214/11-AOP734},
}

@article {MR2964770,
    AUTHOR = {Erd{\H{o}}s, L\'aszl\'o{} and Knowles, Antti and Yau, Horng-Tzer
              and Yin, Jun},
     TITLE = {Spectral statistics of {E}rd{\H{o}}s--{R}\'enyi {G}raphs {II}:
              {E}igenvalue spacing and the extreme eigenvalues},
   JOURNAL = {Commun. Math. Phys.},
  FJOURNAL = {Communications in Mathematical Physics},
    VOLUME = {314},
      YEAR = {2012},
    NUMBER = {3},
     PAGES = {587--640},
      ISSN = {0010-3616,1432-0916},
   MRCLASS = {82B44 (60B20)},
  MRNUMBER = {2964770},
MRREVIEWER = {Longmin\ Wang},
       DOI = {10.1007/s00220-012-1527-7},
       URL = {https://doi-org.emedien.ub.uni-muenchen.de/10.1007/s00220-012-1527-7},
}

@article {MR4691859,
    AUTHOR = {Alt, Johannes and Ducatez, Raphael and Knowles, Antti},
     TITLE = {Localized phase for the {E}rd{\H{o}}s-{R}\'enyi graph},
   JOURNAL = {Commun. Math. Phys.},
  FJOURNAL = {Communications in Mathematical Physics},
    VOLUME = {405},
      YEAR = {2024},
    NUMBER = {1},
     PAGES = {No. 9, 1--74},
      ISSN = {0010-3616,1432-0916},
   MRCLASS = {05C80 (15A18 60B20)},
  MRNUMBER = {4691859},
MRREVIEWER = {Yilun\ Shang},
       DOI = {10.1007/s00220-023-04918-z},
       URL = {https://doi-org.emedien.ub.uni-muenchen.de/10.1007/s00220-023-04918-z},
}

@article {MR4515695,
    AUTHOR = {Alt, Johannes and Ducatez, Raphael and Knowles, Antti},
     TITLE = {Poisson statistics and localization at the spectral edge of
              sparse {E}rd{\H{o}}s-{R}\'enyi graphs},
   JOURNAL = {Ann. Probab.},
  FJOURNAL = {The Annals of Probability},
    VOLUME = {51},
      YEAR = {2023},
    NUMBER = {1},
     PAGES = {277--358},
      ISSN = {0091-1798,2168-894X},
   MRCLASS = {60B20 (05C80 15B52)},
  MRNUMBER = {4515695},
       DOI = {10.1214/22-aop1596},
       URL = {https://doi-org.emedien.ub.uni-muenchen.de/10.1214/22-aop1596},
}

@article {MR4375917,
    AUTHOR = {Alt, Johannes and Ducatez, Rapha\"el and Knowles, Antti},
     TITLE = {The completely delocalized region of the {E}rd{\H{o}}s--{R}\'enyi
              graph},
   JOURNAL = {Electron. Commun. Probab.},
  FJOURNAL = {Electronic Communications in Probability},
    VOLUME = {27},
      YEAR = {2022},
     PAGES = {No. 10, 1--9},
      ISSN = {1083-589X},
   MRCLASS = {60B20 (05C80)},
  MRNUMBER = {4375917},
MRREVIEWER = {Hai\ Yan\ Chen},
       DOI = {10.1214/22-ECP450},
}

@article {MR4328063,
    AUTHOR = {Alt, Johannes and Ducatez, Raphael and Knowles, Antti},
     TITLE = {Delocalization transition for critical {E}rd{\H{o}}s--{R}\'enyi
              graphs},
   JOURNAL = {Commun. Math. Phys.},
  FJOURNAL = {Communications in Mathematical Physics},
    VOLUME = {388},
      YEAR = {2021},
    NUMBER = {1},
     PAGES = {507--579},
      ISSN = {0010-3616,1432-0916},
   MRCLASS = {05C80 (05C50)},
  MRNUMBER = {4328063},
MRREVIEWER = {Yizhe\ Zhu},
       DOI = {10.1007/s00220-021-04167-y},
}

@article {MR4255147,
    AUTHOR = {Alt, Johannes and Ducatez, Rapha\"el and Knowles, Antti},
     TITLE = {Extremal eigenvalues of critical {E}rd{\H{o}}s--{R}\'enyi graphs},
   JOURNAL = {Ann. Probab.},
  FJOURNAL = {The Annals of Probability},
    VOLUME = {49},
      YEAR = {2021},
    NUMBER = {3},
     PAGES = {1347--1401},
      ISSN = {0091-1798,2168-894X},
   MRCLASS = {60B20 (05C80 15B52)},
  MRNUMBER = {4255147},
MRREVIEWER = {Ke\ Wang},
       DOI = {10.1214/20-aop1483},
}

@Article{Gil59,
    AUTHOR = {Gilbert, E. N.},
     TITLE = {Random graphs},
   JOURNAL = {Ann. Math. Statist.},
  FJOURNAL = {Annals of Mathematical Statistics},
    VOLUME = {30},
      YEAR = {1959},
    PAGES = {1141--1144},
  MRNUMBER = {108839},
MRREVIEWER = {John\ Riordan},
       DOI = {10.1214/aoms/1177706098},
 }

@Article{ErRe60,
  author     = {Erd{\H{o}}s, P. and R\'{e}nyi, A.},
  title      = {On the evolution of random graphs},
  journal    = {Magyar Tud. Akad. Mat. Kutat\'{o} Int. K\"{o}zl.},
  year       = {1960},
  volume     = {5},
  pages      = {17--61},
  mrclass    = {05.40},
  mrnumber   = {0125031},
  mrreviewer = {J. Riordan},
}

@book {Bol01,
    AUTHOR = {Bollob\'{a}s, B\'{e}la},
     TITLE = {Random graphs},
   EDITION = {Second},
 PUBLISHER = {Cambridge University Press, Cambridge},
      YEAR = {2001},
     PAGES = {xviii+498},
      ISBN = {0-521-80920-7; 0-521-79722-5},
   MRCLASS = {05C80 (60C05)},
  MRNUMBER = {1864966},
       DOI = {10.1017/CBO9780511814068},
}

@article {KiKhMu06,
    AUTHOR = {Khorunzhiy, Oleksiy and Kirsch, Werner and M\"uller, Peter},
     TITLE = {Lifshitz tails for spectra of {E}rd{\H{o}}s-{R}\'enyi random
              graphs},
   JOURNAL = {Ann. Appl. Probab.},
  FJOURNAL = {The Annals of Applied Probability},
    VOLUME = {16},
      YEAR = {2006},
    NUMBER = {1},
     PAGES = {295--309},
      ISSN = {1050-5164,2168-8737},
   MRCLASS = {05C80 (15A52 60C05)},
  MRNUMBER = {2209343},
MRREVIEWER = {Mark\ R.\ Jerrum},
       DOI = {10.1214/1050516000000719},
       URL = {https://doi-org.emedien.ub.uni-muenchen.de/10.1214/1050516000000719},
}

@misc{Sattler10,
  title = {High-energy asymptotics of Erd{\H{o}}s--R\'enyi random graphs},
  author = {Sattler, C.},
  note = {Diploma thesis, LMU M\"unchen},
  year = {2010},
}

@article{MullerStern24,
 author = {M{\"u}ller, Peter and Stern, Jakob},
 title = {On the return probability of the simple random walk on {Galton}--{Watson} trees},
 fjournal = {Electronic Journal of Probability},
 journal = {Electron. J. Probab.},
 issn = {1083-6489},
 volume = {30},
 pages = {no. 10, 1--23},
  year = {2025},
 language = {English},
 doi = {10.1214/25-EJP1266},
  MRCLASS = {05C81 (60J80 60K37)},
  MRNUMBER = {4852008},
MRREVIEWER = {Huili\ Liu},
 }

@book {HoJo13,
    AUTHOR = {Horn, Roger A. and Johnson, Charles R.},
     TITLE = {Matrix analysis},
   EDITION = {Second},
 PUBLISHER = {Cambridge University Press, Cambridge},
      YEAR = {2013},
     PAGES = {xviii+643},
      ISBN = {978-0-521-54823-6},
   MRCLASS = {15-01},
  MRNUMBER = {2978290},
MRREVIEWER = {Mohammad\ Sal\ Moslehian},
}

@book {Bhatia97,
    AUTHOR = {Bhatia, Rajendra},
     TITLE = {Matrix analysis},
 PUBLISHER = {Springer-Verlag, New York},
      YEAR = {1997},
     PAGES = {xii+347},
      ISBN = {0-387-94846-5},
   MRCLASS = {15-02 (47-02)},
  MRNUMBER = {1477662},
MRREVIEWER = {R.\ J.\ Bumcrot},
       DOI = {10.1007/978-1-4612-0653-8},
       URL = {https://doi-org/10.1007/978-1-4612-0653-8},
}

@article{Jor08,
 author = {Jorgensen, Palle E. T.},
 title = {Essential self-adjointness of the graph-{Laplacian}},
 fjournal = {Journal of Mathematical Physics},
 journal = {J. Math. Phys.},
 issn = {0022-2488},
 volume = {49},
 number = {7},
 pages = {073510-1--33},
 year = {2008},
 language = {English},
 doi = {10.1063/1.2953684},
 zbMATH = {5497144},
 Zbl = {1152.81496}
}

@article{KeLe10,
 author = {Keller, M. and Lenz, D.},
 title = {Unbounded {Laplacians} on graphs: basic spectral properties and the heat equation},
 fjournal = {Mathematical Modelling of Natural Phenomena},
 journal = {Math. Model. Nat. Phenom.},
 issn = {0973-5348},
 volume = {5},
 number = {4},
 pages = {198--224},
 year = {2010},
 language = {English},
 doi = {10.1051/mmnp/20105409},
 url = {https://eudml.org/doc/197705},
 zbMATH = {5762888},
 Zbl = {1207.47032}
}

@book {AthreyaNey04,
    AUTHOR = {Athreya, K. B. and Ney, P. E.},
     TITLE = {Branching processes},
 PUBLISHER = {Dover, Mineola, NY},
      YEAR = {2004},
     PAGES = {xii+287},
      ISBN = {0-486-43474-5},
   MRCLASS = {60-02 (60J10 60J27 60J45 60J80)},
  MRNUMBER = {2047480},
}

@book {Hofstad17,
    AUTHOR = {van der Hofstad, Remco},
     TITLE = {Random graphs and complex networks. {V}ol. 1},
     PUBLISHER = {Cambridge University Press, Cambridge},
      YEAR = {2017},
     PAGES = {xvi+321},
  MRNUMBER = {3617364},
       DOI = {10.1017/9781316779422},
 }

@book{Hofstad24, 
    place={Cambridge}, 
    title={Random graphs and complex networks. {V}ol. 2}, 
    publisher={Cambridge University Press, Cambridge}, 
    author={van der Hofstad, Remco}, 
    year={2024}, 
    doi={10.1017/9781316795552},    
    url={https://doi.org/10.1017/9781316795552}
}

@book {Harris63,
    AUTHOR = {Harris, Theodore E.},
     TITLE = {The theory of branching processes},
 PUBLISHER = {Springer, Berlin},
      YEAR = {1963},
     PAGES = {xiv+230},
   MRCLASS = {60.67},
  MRNUMBER = {163361},
MRREVIEWER = {P.\ A. P. Moran},
}

@article {Piau98,
    AUTHOR = {Piau, Didier},
     TITLE = {Th\'eor\`eme central limite fonctionnel pour une marche au
              hasard en environnement al\'eatoire},
   JOURNAL = {Ann. Probab.},
  FJOURNAL = {The Annals of Probability},
    VOLUME = {26},
      YEAR = {1998},
    NUMBER = {3},
     PAGES = {1016--1040},
      ISSN = {0091-1798,2168-894X},
   MRCLASS = {60J15 (60F17 60J80 60K35)},
  MRNUMBER = {1634413},
MRREVIEWER = {H.\ Kesten},
       DOI = {10.1214/aop/1022855743},
       URL = {https://doi.org/10.1214/aop/1022855743},
}

@Book{LyonsPeres16,
 Author = {Lyons, Russell and Peres, Yuval},
 Title = {Probability on trees and networks},
 ISBN = {978-1-107-16015-6; 978-1-108-73272-7; 978-1-316-67281-5},
 Year = {2016},
 Language = {English},
 DOI = {10.1017/9781316672815},
 PUBLISHER = {Cambridge University Press, New York},
     PAGES = {xv+699},
  MRNUMBER = {3616205},
MRREVIEWER = {Laurent\ Miclo},
}

@Article{KesteStigu66,
 Author = {Kesten, H. and Stigum, B. P.},
 Title = {A limit theorem for multidimensional {Galton}--{Watson} processes},
 FJournal = {Annals of Mathematical Statistics},
 Journal = {Ann. Math. Statist.},
 ISSN = {0003-4851},
 Volume = {37},
 Pages = {1211--1223},
 Year = {1966},
 Language = {English},
 MRNUMBER = {198552},
    MRREVIEWER = {I.\ J.\ Good},
 DOI = {10.1214/aoms/1177699266},
   }

@article{ChenPeres,
 Author = {Chen, D. and Peres, Y.},
 Title = {Anchored expansion, percolation and speed. {W}ith an appendix by
{G}ábor {P}ete},
 FJournal = {Annals of Probability},
 Journal = {Ann. Probab.},
 Volume = {32},
year={2004},
     PAGES = {2978--2995},
      ISSN = {0091-1798,2168-894X},
       DOI = {10.1214/009117904000000586},
       URL = {https://doi.org/10.1214/009117904000000586},
    NUMBER = {4},
  MRNUMBER = {2094436},
MRREVIEWER = {Timo\ Sepp\"al\"ainen},
}

@article {legall05,
    AUTHOR = {Le Gall, Jean-Fran\c{o}is},
     TITLE = {Random trees and applications},
   JOURNAL = {Probab. Surv.},
  FJOURNAL = {Probability Surveys},
    VOLUME = {2},
      YEAR = {2005},
     PAGES = {245--311},
      ISSN = {1549-5787},
   MRCLASS = {60J80 (05C05 05C80 35J65 60C05 60J65)},
  MRNUMBER = {2203728},
MRREVIEWER = {Endre\ Cs\'aki},
       DOI = {10.1214/154957805100000140},
       URL = {https://doi.org/10.1214/154957805100000140},
}

@article {Aidekon14,
    AUTHOR = {A{\"{i}}d\'ekon, Elie},
     TITLE = {Speed of the biased random walk on a {G}alton-{W}atson tree},
   JOURNAL = {Probab. Theory Relat. Fields},
  FJOURNAL = {Probability Theory and Related Fields},
    VOLUME = {159},
      YEAR = {2014},
    NUMBER = {3-4},
     PAGES = {597--617},
      ISSN = {0178-8051,1432-2064},
   MRCLASS = {60J80 (60F05 60G50)},
  MRNUMBER = {3230003},
MRREVIEWER = {Jyy-I\ Hong},
       DOI = {10.1007/s00440-013-0515-y},
       URL = {https://doi.org/10.1007/s00440-013-0515-y},
}

@incollection {BenArousFribergh16,
    AUTHOR = {Ben Arous, G\'erard and Fribergh, Alexander},
     TITLE = {Biased random walks on random graphs},
 BOOKTITLE = {Probability and statistical physics in {S}t. {P}etersburg},
    SERIES = {Proc. Sympos. Pure Math.},
    VOLUME = {91},
     PAGES = {99--153},
 PUBLISHER = {Amer. Math. Soc., Providence, RI},
      YEAR = {2016},
      ISBN = {978-1-4704-2248-6},
   MRCLASS = {60K37 (05C80 05C81 60J80)},
  MRNUMBER = {3526827},
MRREVIEWER = {Yueyun\ Hu},
       DOI = {10.1090/pspum/091/01536},
       URL = {https://doi.org/10.1090/pspum/091/01536},
}

@article {LyonsPemantlePeres96,
    AUTHOR = {Lyons, Russell and Pemantle, Robin and Peres, Yuval},
     TITLE = {Biased random walks on {G}alton-{W}atson trees},
   JOURNAL = {Probab. Theory Relat. Fields},
  FJOURNAL = {Probability Theory and Related Fields},
    VOLUME = {106},
      YEAR = {1996},
    NUMBER = {2},
     PAGES = {249--264},
      ISSN = {0178-8051,1432-2064},
   MRCLASS = {60J80 (60J15)},
  MRNUMBER = {1410689},
MRREVIEWER = {Ben\ Hambly},
       DOI = {10.1007/s004400050064},
       URL = {https://doi.org/10.1007/s004400050064},
}

@article {DemboEtAl02,
    AUTHOR = {Dembo, Amir and Gantert, Nina and Peres, Yuval and Zeitouni,
              Ofer},
     TITLE = {Large deviations for random walks on {G}alton-{W}atson trees:
              averaging and uncertainty},
   JOURNAL = {Probab. Theory Relat. Fields},
  FJOURNAL = {Probability Theory and Related Fields},
    VOLUME = {122},
      YEAR = {2002},
    NUMBER = {2},
     PAGES = {241--288},
      ISSN = {0178-8051,1432-2064},
   MRCLASS = {60K37 (60F10 60J80 82B41 82B44)},
  MRNUMBER = {1894069},
MRREVIEWER = {Boualem\ Djehiche},
       DOI = {10.1007/s004400100162},
       URL = {https://doi.org/10.1007/s004400100162},
}
	\bibliographystyle{mybib} 
	\vspace{1cm}
	
\end{document}